\documentclass[11pt]{article}

\usepackage[T1]{fontenc}
\usepackage[utf8]{inputenc}
\usepackage[english]{babel}
\usepackage{lmodern}
\usepackage{microtype}
\usepackage[a4paper,margin=1.15in]{geometry}
\usepackage{setspace}
\AtBeginDocument{%
  \setlength{\abovedisplayskip}{0.28\baselineskip}%
  \setlength{\belowdisplayskip}{0.28\baselineskip}%
  \setlength{\abovedisplayshortskip}{0.14\baselineskip}%
  \setlength{\belowdisplayshortskip}{0.14\baselineskip}%
}
\usepackage{fancyhdr}
\usepackage{graphicx}
\graphicspath{{./Images/}}
\usepackage{subcaption}
\usepackage[export]{adjustbox}
\usepackage{array}
\usepackage{booktabs}
\usepackage{longtable}
\usepackage{enumitem}
\usepackage{xcolor}
\usepackage{comment}
\usepackage{etoolbox}
\usepackage{amsmath,amssymb,amsthm,amsfonts}
\usepackage{mathtools}
\usepackage{bbm}
\usepackage{imakeidx}
\usepackage{wasysym}
\usepackage{soul}
\usepackage{mdframed}
\usepackage{matlab-prettifier}
\usepackage{titlesec}
\titlespacing*{\section}{0pt}{1em}{0.5em}
\titlespacing*{\subsection}{0pt}{0.5em}{0.25em}
\usepackage{datetime}
\usepackage[colorlinks=true,linkcolor=blue,citecolor=blue,urlcolor=blue]{hyperref}
\hypersetup{
  colorlinks=true,
  linkcolor=blue,
  citecolor=blue,
  urlcolor=blue,
  pdfnewwindow=true
}
\usepackage[nameinlink,noabbrev]{cleveref}

\newcommand{\R}{\mathbb{R}}

\DeclareMathOperator{\diver}{div}

\newcommand{\eps}{\varepsilon}

\DeclareMathOperator{\dist}{dist}
\DeclareMathOperator{\supp}{supp}

\theoremstyle{plain}
\newtheorem{theorem}{Theorem}[section]
\newtheorem{proposition}[theorem]{Proposition}
\newtheorem{lemma}[theorem]{Lemma}
\newtheorem{corollary}[theorem]{Corollary}

\theoremstyle{definition}
\newtheorem{definition}[theorem]{Definition}
\newtheorem{example}[theorem]{Example}

\theoremstyle{remark}
\newtheorem{remark}[theorem]{Remark}
\newtheorem{assumption}[theorem]{Assumption}
\newtheorem*{notation}{Notation}

\colorlet{red}{black}
\hypersetup{linkcolor=black,citecolor=black,urlcolor=black}

\makeatletter
\renewcommand{\@fnsymbol}[1]{\ifcase#1\or 1\or *\or \dagger\or \ddagger\or \mathsection\or \mathparagraph\or \|\or **\or \dagger\dagger\or \ddagger\ddagger\else\@ctrerr\fi}
\makeatother

\title{State--constrained optimal control of the continuity equation and infinite--dimensional viscosity solutions}
\author{%
Fabio Bagagiolo\thanks{Department of Mathematics, University of Trento, Via Sommarive 14, Povo, TN 38123, Italy.}
\and
Ivan Roman\`o\textsuperscript{1}\thanks{Corresponding author: Ivan Roman\`o; email: \href{mailto:ivan.romano@unitn.it}{ivan.romano@unitn.it}; ORCID: \href{https://orcid.org/0009-0003-1389-1543}{0009-0003-1389-1543}.}}
\date{}

\begin{document}
\maketitle

\begin{abstract}
\noindent We study a finite-horizon optimal control problem for the continuity equation under a weighted integral state constraint on the mass outside a fixed set. The model is cast in a Hilbert framework for densities. On a suitable invariant compact subset, we prove that the value function is Lipschitz continuous and satisfies, by dynamic programming, the associated infinite-dimensional constrained Hamilton--Jacobi--Bellman equation in viscosity sense (subsolution in the interior, supersolution up to the boundary). We finally prove a comparison principle and uniqueness in the Lipschitz class.\\

\end{abstract}

\section{Introduction}
\label{sec:introduction}

In this paper we address an optimal control problem for a first-order continuity equation of the form
\begin{equation}
\label{eq:intro0}
\partial_t m+\mbox{div}(\alpha m)=0
\end{equation}
\noindent
where the unknown $m:\mathbb{R}^n\times[0,T]\to\mathbb{R}$ is the density function of the distribution of a mass evolving on $\mathbb{R}^n$ and the vector field $\alpha:\mathbb{R}^n\times[0,T]\to\mathbb{R}^n$ is the control at our disposal. In particular, besides the minimization of an integral cost functional of the form
\[
J(\bar{m},t,\alpha)=\int_t^T\ell(m(\cdot,s),\alpha(\cdot,s))\;ds+\psi(m(\cdot,T)),
\]
\noindent
the optimal control problem is subject to a state-constraint: the mass should remain inside a fixed domain $\Omega\subseteq\mathbb{R}^n$. A suitable way to analytically write such a constraint is also discussed in the paper and the main results involve a relaxed formulation of it of the form
\begin{equation}
\label{eq:intro0.5}
\int_{\mathbb{R}^n\setminus\Omega}p(x)m(x,s)\;dx\le\delta\qquad  \forall s\in[t,T],
\end{equation}
\noindent
where $p:\mathbb{R}^n\to\mathbb{R}$ is a fixed positive penalty function and $\delta>0$ is a fixed threshold. For the value function
\[
V(m,t)=\inf_{\alpha}J(m,t,\alpha),
\]
\noindent
defined on a suitable subset of $L^2(\mathbb{R}^n)\times[0,T]$, we prove, under suitable hypotheses, its (Lipschitz) continuity and its uniqueness as a viscosity solution of the constrained infinite-dimensional Hamilton--Jacobi--Bellman equation, coupled with a terminal condition,
\begin{equation}
\label{eq:intro1}
\left\{
\begin{array}{ll}
\displaystyle
-\partial_tV+\sup_a\left\{\langle \nabla_mV,\mbox{div}(am)\rangle-\ell(m,a)\right\}=0&\mbox{if } \displaystyle \int_{\mathbb{R}^n\setminus\Omega}p(x)m(x)dx<\delta,\\
\displaystyle
-\partial_tV+\sup_a\left\{\langle \nabla_m V,\mbox{div}(am)\rangle-\ell(m,a)\right\}\ge0&\mbox{if } \displaystyle \int_{\mathbb{R}^n\setminus\Omega}p(x)m(x)dx=\delta,\\
\displaystyle
V(m,T)=\psi(m)
\end{array}
\right.
\end{equation}
\noindent
where $a:\mathbb{R}^n\to\mathbb{R}^n$ is a constant (in time) vector field, $\nabla_mV$ is the Fr\'echet differential with respect to $m\in L^2(\mathbb{R}^n)$ and $\langle\cdot,\cdot\rangle$ denotes the scalar product in $L^2(\mathbb{R}^n)$.
We suitably adapt the seminal results for the finite-dimensional case of Soner \cite{Soner1986StateConstraintI} (see also Bardi-Capuzzo Dolcetta \cite{BardiCapuzzoDolcetta1997}) to our infinite-dimensional case. In particular, we first construct a "pushback" control which allows us to compress the mass towards $\Omega$ when the constraint is going to be violated; suitably using such a control as a needle, only when necessary, we prove the Lipschitz continuity of the value function; by the dynamic programming formula we prove that the value function solves, in the viscosity sense, the infinite-dimensional problem \eqref{eq:intro1} in a suitable subset $\bar{X}_c\subseteq L^2(\mathbb{R}^n)\times[0,T]$; constructing a penalization function in the double variable technique, exploiting a suitable inward-cone property of the constraint, we get a comparison result for \eqref{eq:intro1} and hence uniqueness in the class of Lipschitz functions.

As in Bagagiolo-Capuani-Marzufero \cite{bagagiolo2024single,bagagioloCapuaniMarzufero2025zerosum}, where some differential games for \eqref{eq:intro0} are studied, we aim to pose our problem in a Hilbert setting, which means that we are going to work with masses whose distribution over $\mathbb{R}^n$ is given by measures $\mu$ having a density function $m$. In order to keep this true for all times during the evolution, we assume a suitable (spatial) regularity of the initial datum $m_0\in\ L^2(\mathbb{R}^n)$ and of the vector fields $\alpha(\cdot,t):\mathbb{R}^n\to\mathbb{R}^n$. Such regularities also bring the necessary compactness property that we need in order to get the comparison result for (\ref{eq:intro1}), in particular the compactness of $\bar{X}_c$ in $H^1(\mathbb{R}^n)\times[0,T]$. A crucial point is also the estimate between the $H^1$ and the $L^2$ norms. \textcolor{red}{The compactness is also in order to fit the standard analytical framework in finite dimension; this can probably be weakened by using variational principles, but the aim of this work is, other than the search of the maximal generality, the study of the constrained HJB \eqref{eq:intro1} in an infinite-dimensional Hilbert space.}

Several works about the control of the evolution of (\ref{eq:intro0}) have been published in recent years. A main impulse to such research was given by the analysis of first-order mean-field games (see Cardaliaguet \cite{Cardaliaguet2013NotesMFG} and Lasry--Lions \cite{LasryLions2007MFGI}). We point out that in our model the equation \eqref{eq:intro0} is controlled in a centralized manner: the control $\alpha$ is externally chosen by a sort of "supervisor" and not constructed by the individual decisions of the single particles of the mass (decision-making agents). Recent results on centralized control of continuity equations can be grouped into two directions: a DPP/HJB-viscosity approach in Wasserstein spaces (Jimenez--Marigonda--Quincampoix \cite{JimenezMarigondaQuincampoix2020MultiagentWasserstein}, centralized multi-agent formulation and HJB characterization; Aussedat--Jerhaoui--Zidani \cite{AussedatJerhaouiZidani2024ViscosityCentralizedMeasure}, comparison and uniqueness for viscosity solutions in measure spaces) and a variational/Pontryagin approach (Pogodaev \cite{Pogodaev2016OptimalControlContinuityEq}, first-order necessary conditions for controlled continuity equations; Bonnet--Rossi \cite{BonnetRossi2019PMPwasserstein}, Pontryagin maximum principle in the Wasserstein space). In the specific setting of infinite-dimensional control problems with state constraints, Daudin \cite{Daudin2023FokkerPlanckStateConstraints} addresses a second-order (diffusive) Fokker--Planck dynamics in Wasserstein space where the constraint is actually formulated as in (\ref{eq:intro0.5}), Bonnet \cite{Bonnet2019PMPConstrainedWasserstein} proves a Pontryagin maximum principle for constrained problems with nonlocal drift dependence, and Cavagnari--Marigonda--Quincampoix \cite{CavagnariMarigondaQuincampoix2021CompatibilityConstraints} study compatibility conditions between state constraints and admissible multi-agent dynamics for the continuity equation. In all these works equation (\ref{eq:intro0}) is meant for measures and the whole problem is cast in the Wasserstein setting. As we said, in the present paper we adopt a functional perspective and we believe that such an approach, especially for the state-constrained problem, is conceptually natural: working in $L^2$ keeps the density description explicit, provides a Hilbert structure well suited to compactness, stability, and HJB-viscosity analysis \textcolor{red}{and also may allow the use of specific costs to model penalization on density (for the latter, see Example \ref{example})}. Moreover, up to our knowledge, in the measure/Wasserstein setting, the analysis of the continuity of the value function and the uniqueness of the viscosity solution for the corresponding constrained HJB equation arising from the constrained optimal control of (\ref{eq:intro0}) and made à la Soner \cite{Soner1986StateConstraintI} is still lacking in the literature. It may be the subject of future work. A related constrained problem in Hilbert spaces was studied by Cannarsa--Gozzi--Soner \cite{CannarsaGozziSoner1991HilbertHJB}, for a different state equation: they consider an abstract controlled evolution equation and the associated boundary value problem for an infinite-dimensional HJB equation. \textcolor{red}{We also refer to Frankowska-Marchini-Mazzola \cite{frank2016,frank2018} for other related specific results for constrained evolutions in infinite-dimensional Banach spaces.}

We believe that, from the analytical point of view, the novelty of our study mainly relies on the fact that, 
up to our knowledge, it is the first time that 
a state-constrained optimal control problem for the specific first-order equation (\ref{eq:intro0}) is studied directly with regard to the continuity of the value function and a comparison principle for the corresponding infinite-dimensional Hamilton-Jacobi-Bellman problem in a Hilbert space. Note that in Cannarsa-Gozzi-Soner \cite{CannarsaGozziSoner1991HilbertHJB} the state equation, partly motivated by the control of a second-order-type equation (typically the heat equation), is mostly a generalization to the Hilbert case of the finite-dimensional controlled evolution $y'=f(y,u)$, and this fact somehow reflects in all their work. Moreover, that model does not seem to take into account the presence of the control in the principal part of the differential equation as in our case. Finally, we also point out that in \cite{CannarsaGozziSoner1991HilbertHJB} the control problem is of the so-called exit-time type (strongly related to the constraint problem and raising most of the same analytical questions) and hence their value function does not depend on time, unlike ours.

From the application point of view, our model of state-constrained optimal control may arise from many practical situations where a super-manager has to optimally govern the motion of a huge quantity of some kind of elements, be they particles, agents, information, data to be transmitted or to be clustered. Here, just as an example, we quote the recent work by Geshkovski-Letrouit-Polyanskiy-Rigollet \cite{gesletpolrig} (see also Bruno-Pasqualotto-Agazzi \cite{brupasaga}) where, concerning the evolution of interacting particles (tokens, in particular), the authors study a continuity equation as (\ref{eq:intro0}) constrained to the sphere of $\mathbb{R}^n$, where the velocity field is a given function of the solution somehow linked to the optimization of an interaction energy functional. Other possible applications may arise in economic models. See for example Mehadoui-Lacitignola-Tilioua \cite{mehlactil} where a state-constrained optimal control problem for a system of PDEs is studied with the state constraint represented by an integral bound (similar to \eqref{eq:intro0.5}) on the spatially distributed technological progress. 

As generic references for the theory of viscosity solutions for Hamilton-Jacobi equations, both in finite and infinite dimensions, we refer the reader to the seminal papers by Crandall--Lions \cite{CrandallLions1985HJInfiniteDimI,CrandallLions1986HJInfiniteDimII} and to the book by Bardi--Capuzzo Dolcetta \cite{BardiCapuzzoDolcetta1997}, whereas for the analysis of the continuity equation \eqref{eq:intro0}, we refer to Santambrogio \cite{Santambrogio2015OptimalTransport} and Ambrosio--Crippa \cite{ambrosio2008flow,Ambrosio_Crippa_2014}.\\ 

The paper is organized as follows. In Section 2 we study the controlled continuity equation \eqref{eq:intro0}, show regularity and stability estimates, and construct the compact invariant set used in the analysis. In Section 3 we introduce the state-constrained optimal control problem, prove the dynamic programming principle, and establish the Lipschitz continuity of the value function on $\bar{X}_c\subseteq L^2(\R^n)\times [0,T]$. In Section 4 we derive the corresponding constrained Hamilton--Jacobi--Bellman equation and show that the value function is a constrained viscosity solution. In Section 5 we prove the comparison principle and deduce uniqueness in the Lipschitz class. In Section 6 we discuss alternative formulations of the state constraint and comment on their suitability for our framework. The Appendix collects various estimates on the flow associated to the continuity equation, mostly used in Section 2.

\section{The continuity equation}
\label{sec:continuityeq}

In this Section we first recall the flow representation and some uniform estimates; regularity and stability are proved in Subsection~\ref{subsec:traj-regularity}, and we construct an invariant compact set $X$ for the evolutions in Subsection~\ref{subsec:construction of an invariant}.\\

\noindent Let $T$ be the fixed time horizon, $t\in [0,T)$  and consider a time-dependent (velocity) vector field $\alpha:\R^n \times [0,T] \to \R^n$. The corresponding evolution in time of a density $m:\R^n\times[t,T]\to \R$ is described by the continuity equation initial-data problem \eqref{eq:continuity}.
\begin{equation}\label{eq:continuity} 
    \begin{cases}
        \partial_s m(x,s) + \diver(\alpha(x,s)m(x,s)) = 0& \text{for }(x,s)\in\mathbb{R}^n\times (t,T),\\
        m(x,t) = \bar{m}(x)& \text{for }x\in\R^n,
    \end{cases}
\end{equation}

\noindent where $\bar{m}:\R^n\to \R$ is the initial density at time $t$.\\
The flow associated with the system \eqref{eq:continuity}, denoted by $\Phi$, is the solution of 
\begin{equation}\label{eq:flow equation}
    \begin{cases}
        \partial_s \Phi(x,s,t) = \alpha(\Phi(x,s,t),s), & s\in(t,T),\\
        \Phi(x,t,t) = x.
    \end{cases}
\end{equation}
for every $x\in\R^n$, where $\Phi(x,s,t) $ is the position at time $s$ of a particle that has started at time $t$ from the position $x$, under the velocity field $\alpha$.

\begin{assumption}\label{ass:alpha}
    We assume that the vector field $\alpha$ belongs to 
    \[
        \mathcal{A}:=\left\{\alpha \in L^2\left([0,T],W^{3,\infty}(\R^n;\R^n)\cap H^1(\R^n;\R^n)\right):\alpha(\cdot,s)\in \tilde{\mathcal{A}} \text{ for a.e. }\;s\in [0,T]\right\},
    \]
    where, for a fixed constant $M>0$,
    \[
        \tilde{\mathcal{A}}:=\left\{ a\in W^{3,\infty}(\R^n;\R^n)\cap H^1(\R^n;\R^n): \max\left\{ \|a\|_{W^{3,\infty}},\|a\|_{H^1} \right\}\leq M \right\}.
    \]
\end{assumption}

\begin{remark}\label{rem:uniformflowestimates}
    In the sequel we will use the notation $\Phi_{t,s}:=\Phi(\cdot,s,t):\R^n\to \R^n$ for the flow and $\Psi_{t,s}:=\Phi_{t,s}^{-1}$ for its inverse. We refer to Appendix~\ref{app:flow-estimates} (Proposition \ref{prop:appendix}) for details on the flow and its inverse, in particular for estimates on them and their derivatives, which work due to the spatial regularity Assumption \ref{ass:alpha}.
    Moreover, see for example \cite{Ambrosio_Crippa_2014}, the unique explicit solution of \eqref{eq:continuity}, denoted by $m(\cdot,\cdot;\alpha,\bar{m},t)$, is the pushforward of $\bar{m}$ by the flow map $\Phi_{t,s}$, i.e., \begin{equation}\label{eq:representation}
        m(x,s;\alpha,\bar{m},t):= \frac{\bar{m}(\Psi_{t,s}(x))}{\det D\Phi_{t,s}(\Psi_{t,s}(x))}\qquad \forall (x,s)\in \R^n\times [t,T],
    \end{equation}
    where $D$ stays for the spatial Jacobian matrix. The solution of \eqref{eq:continuity} is understood in the distributional sense, even if we will see that under our assumptions it will be more regular.
\end{remark}

\noindent 

\subsection{Regularity of the solutions}\label{subsec:traj-regularity}

In this subsection we establish regularity and stability estimates for the solutions of \eqref{eq:continuity}, under the Assumption \ref{ass:alpha}.

\begin{proposition}\label{prop:boundW2infty}
    There exists a constant $C=C(M,T,n)$ such that, for every $t\in[0,T)$, every $\bar m\in W^{2,\infty}(\R^n)$, and every $\alpha\in\mathcal{A}$, one has
    \[
        \|m(\cdot,s;\alpha,\bar m,t)\|_{W^{2,\infty}}\leq C\|\bar m\|_{W^{2,\infty}}
        \qquad \forall s\in[t,T].
    \]
   An analogous estimate also holds for the $H^2$ norm, starting from an initial datum in $H^2(\R^n)$.
\end{proposition}
\begin{proof}
    Invoking the notation introduced in Proposition \ref{prop:appendix}, setting $m(\cdot,s):=m(\cdot,s;\alpha,\bar m,t)$ we easily get $\|m(\cdot,s)\|_{L^\infty}\le \|\bar m\|_{L^\infty}\,\|J_{t,s}\|_{L^\infty}\le e^{nMT}\|\bar m\|_{L^\infty}$, by \ref{eq:flowregularity5}.
    \noindent Furthermore, recalling in particular \ref{eq:flowregularity2} and \ref{eq:flowregularity7}, 
    \begin{align*}
        \|\nabla m(\cdot,s)\|_{L^\infty}
        &\le \|\nabla \bar m\circ\Psi_{t,s}\|_{L^\infty}\,\|D\Psi_{t,s}\|_{L^\infty}\,\|J_{t,s}\|_{L^\infty}
        + \|\bar m\circ\Psi_{t,s}\|_{L^\infty}\,{\color{red}\|D\Psi_{t,s}\|_{L^\infty}}\,\|\nabla J_{t,s}\|_{L^\infty} \\
        &\le (e^{MT}+1)\,\tilde{C}_J\,\|\bar m\|_{W^{1,\infty}},
    \end{align*}
    and proceeding analogously, using the uniform boundedness of $\|D^2\Psi_{t,s}\|_{L^\infty}$ together with \ref{eq:flowregularity8}, we obtain $ \|D^2 m(\cdot,s)\|_{L^\infty}\le C_H\,\|\bar m\|_{W^{2,\infty}}$,
    where $C_H=C(M,T,n)$. Combining the previous estimates yields the desired bound in $W^{2,\infty}$.\\ 
    For the $H^2$-estimate, for example,
    \begin{align*}
        \|\nabla m(\cdot,s)\|_{L^2}
        &\le \|\nabla \bar m\circ\Psi_{t,s}\|_{L^2}\,\|D\Psi_{t,s}\|_{L^\infty}\,\|J_{t,s}\|_{L^\infty}
        + \|\bar m\circ\Psi_{t,s}\|_{L^2}\,{\color{red}\|D\Psi_{t,s}\|_{L^\infty}}\,\|\nabla J_{t,s}\|_{L^\infty} \\
        &\le (e^{MT}+1)\,\tilde{C}_J\,\|\bar m\|_{H^1}.
    \end{align*}
    by an easy change of variables; the estimates for $\|m(\cdot,s)\|_{L^2}$ and $\|D^2 m(\cdot,s)\|_{L^2}$ follows in the same spirit.
\end{proof}

\begin{proposition}\label{prop:useful}
    The following estimates hold:
    \begin{enumerate}
        \item[i)] for all $t\in [0,T),\;\bar{m}\in H^2(\R^n)$, $\alpha\in\mathcal{A}$ and $ s_1,s_2\in [t,T] $ we have that
        \begin{equation*}
            \|m(\cdot,s_1;\alpha,\bar{m},t)-m(\cdot,s_2;\alpha,\bar m,t)\|_{H^1} \leq C(M,n)\left[\sup_{\tau\in (s_1,s_2)} \|m(\cdot,\tau;\alpha,\bar m,t)\|_{H^2}\right]|s_1-s_2|; 
        \end{equation*}
        \item[ii)] for all $t_1,t_2\in [0,T),\;m_1,m_2\in H^2(\R^n),$ $\alpha\in\mathcal{A}$ and $s\in [\max\{t_1,t_2\},T]$ we have that
        \begin{equation*}
            \|m(\cdot,s;\alpha,m_1,t_1)-m(\cdot,s;\alpha,m_2,t_2)\|_{L^2} \leq \tilde{L}(\|m_1-m_2\|_{L^2}+|t_1-t_2|),
        \end{equation*}
        where $\tilde{L}$ depends only on $M,n,T$ and on the $H^2$ norm of the evolution associated with $(m_1,t_1,\alpha)$. Moreover, the same estimate holds with the $L^2$ norm replaced by the $H^1$ norm; see \cite[Proposition~1]{bagagioloCapuaniMarzufero2025zerosum}. 
    \end{enumerate} 
\end{proposition}

\begin{proof}
    \emph{i)} Since $m:=m(\cdot,\cdot;\alpha,\bar m,t)$ solves the continuity equation, we have
    $\partial_s m(\cdot,s)=-\diver(\alpha(\cdot,s)m(\cdot,s))$ for a.e.\ $s\in [t,T]$; moreover, using Proposition~\ref{prop:boundW2infty} and the uniform bounds on $\alpha\in\mathcal{A}$, it follows that $\partial_s m\in L^1([t,T];H^1(\R^n))$, and hence $m\in W^{1,1}([t,T];H^1(\R^n))$. Therefore, 
    \begin{align*}
        \|m(\cdot,s_2)-m(\cdot,s_1)\|_{H^1}
        &\le \int_{s_1}^{s_2}\|\diver(\alpha(\cdot,\tau)m(\cdot,\tau))\|_{H^1}\,d\tau\\
        &\le \int_{s_1}^{s_2}\Bigl(
        \|\alpha(\cdot,\tau)\cdot\nabla m(\cdot,\tau)\|_{L^2}
        +\|m(\cdot,\tau)\,\diver\alpha(\cdot,\tau)\|_{L^2}\\
        &\hspace{3.25em}
        +\|\nabla(\alpha(\cdot,\tau)\cdot\nabla m(\cdot,\tau))\|_{L^2}
        +\|\nabla(m(\cdot,\tau)\,\diver\alpha(\cdot,\tau))\|_{L^2}
        \Bigr)\,d\tau\\
        &\le C(M,n)\int_{s_1}^{s_2}\|m(\cdot,\tau)\|_{H^2}\,d\tau,
    \end{align*}
    and the conclusion follows.\\ 
    \emph{ii)} Assume $t_1\le t_2$. Define $\bar m:=m(\cdot,t_2;\alpha,m_1,t_1)$. Then, for every $s\in [\max\{t_1,t_2\},T]$,
    \begin{equation*}\label{eq:ii1}
        \|m(\cdot,s;\alpha,m_1,t_1) - m(\cdot,s;\alpha,m_2,t_2)\|_{L^2} = \|m(\cdot,s;\alpha,\bar{m},t_2) - m(\cdot,s;\alpha,m_2,t_2)\|_{L^2}=:\|u(\cdot,s)\|_{L^2}.
    \end{equation*}
    Note that both densities satisfy \eqref{eq:continuity} with the same vector field; then it follows from \eqref{eq:representation}, using \ref{eq:flowregularity5}, that (using $\Psi_{t_2,s}$ as a change of variables)
    \begin{align*}
      \|u(\cdot,s)\|_{L^2}^2 & = \int_{\R^n} \left[\frac{\bar m(\Psi_{t_2,s}(x)) -m_2(\Psi_{t_2,s}(x)) }{\det D\Phi_{t_2,s}(\Psi_{t_2,s}(x))}\right]^2\;dx\\
      & = \int_{\R^n} \frac{(\bar m(y)-m_2(y))^2}{\det D\Phi_{t_2,s}(y)}\;dy \leq e^{nMT}\|\bar{m}-m_2\|_{L^2}^2.
    \end{align*}
    Therefore, using item~i), we obtain
    \begin{align*}
        \|u(\cdot,s)\|_{L^2} & \leq e^{\frac{nMT}{2}} \left[ \|\bar{m}-m_1\|_{L^2} + \|m_1-m_2\|_{L^2} \right]\\
        & \leq e^{\frac{nMT}{2}} \left[ C(M,n)\left( \sup_{\tau\in (t_1,t_2)}\|m(\cdot,\tau;\alpha,m_1,t_1)\|_{H^2} \right)|t_1-t_2| + \|m_1-m_2\|_{L^2} \right]
    \end{align*} 
\end{proof}

\begin{proposition}\label{prop:energy-ex}
    There exists a constant $C = C(M,T,n)$ such that for every $m:\R^n\times [t,T]\to \R$ a solution of \eqref{eq:continuity} with initial datum $\bar{m}\in H^1(\R^n)$ and velocity field $\alpha\in\mathcal{A}$, then
    \[  
        \|m(\cdot,s)\|_{H^1} \leq e^{C(s-t)}\|\bar{m}\|_{H^1}\qquad \forall s\in [t,T],
    \]
\end{proposition}

\begin{proof}
    With estimates similar to those in the proof of Proposition~\ref{prop:useful}, item~ii), one shows that
    $\displaystyle \|m(\cdot,s)\|_{L^2} \leq e^{\frac{nM}{2}(s-t)}\|\bar{m}\|_{L^2}$. Moreover, using \eqref{eq:representation} and a direct computation,
    \[  
        \left\|\nabla m(\cdot,s)\right\|_{L^2} \leq \left\| \frac{\nabla \bar{m}\circ \Psi_{t,s}}{\det D\Phi_{t,s}\circ \Psi_{t,s}}  \right\|_{L^2}\|D\Psi_{t,s}\|_{L^\infty} + \left\| \bar{m}\circ \Psi_{t,s} \right\|_{L^2} \left\|\nabla\left( \det D\Phi_{t,s}\circ \Psi_{t,s} \right)^{-1} \right\|_{L^\infty}.
    \]
    Now, the $L^\infty$-norm terms are uniformly bounded by Proposition~\ref{prop:appendix} (in particular \ref{eq:flowregularity2} and \ref{eq:flowregularity7}), while the $L^2$-norm terms, by the usual change of variables, are estimated respectively by terms of the form
    $e^{C_1(s-t)}\|\nabla \bar{m}\|_{L^2}$ and $e^{C_2(s-t)}\|\bar{m}\|_{L^2}$ with uniform constants $C_1,C_2>0$. Summing the contributions, and possibly adjusting the constants, yields the claim.
\end{proof}

\subsection{An invariant and compact set for the evolutions}\label{subsec:construction of an invariant}

Following an argument similar to the one in \cite{bagagioloCapuaniMarzufero2025zerosum}, we construct a compact and invariant set $X\subseteq H^1(\R^n)\times [0,T]$ for the evolution of the masses.
\begin{assumption}\label{ass26}
    Let $\tilde{\Omega}\subset \R^n$ be open and bounded and let $C_U>0$. We assume that the masses at time $t=0$ are non-negative densities that belong to 
    \[
        X(0):=\left\{ \bar{m}\in W^{2,\infty}(\R^n) : \max\{ \|\bar{m}\|_{W^{2,\infty}},\|\bar{m}\|_{H^1} \}\leq C_U,\;\supp(\bar{m})\subseteq \tilde{\Omega},\;{\color{red}\bar m\geq 0 }\right\}.
    \]
\end{assumption}

\noindent It is clear that for every $\bar{m}\in X(0)$ and $\alpha \in \mathcal{A}$ one has
$\supp\big(m(\cdot,t;\alpha,\bar{m},0)\big)\subseteq \tilde{\Omega}(t)$ for all $t\in (0,T]$, where, for $t\geq 0$,
\[
    \tilde{\Omega}(t):=\left\{ x\in \R^n : \dist (x,\tilde{\Omega})\leq Mt \right\}.
\]
Moreover, by Proposition \ref{prop:boundW2infty}, for every $\bar{m}\in X(0)$ and $\alpha \in \mathcal{A}$,
\[
    \|m(\cdot,t;\alpha,\bar{m},0)\|_{W^{2,\infty}} \leq C\|\bar{m}\|_{W^{2,\infty}} \leq C C_U =: \tilde{C}_U\qquad \forall t\in [0,T]. 
\]
We then define, for each $t\in (0,T]$, the set
\begin{alignat*}{2}
    X(t):= \Big\{ \bar{m}\in W^{2,\infty}(\R^n) : && \|m(\cdot,s;\alpha,\bar m,t)\|_{W^{2,\infty}}\leq \tilde{C}_U\;\forall s\in [t,T],\;\forall \alpha\in\mathcal{A}, & \\
    && \supp(\bar{m})\subseteq \tilde{\Omega}(t), \;{\color{red}\bar m \geq 0} & \Big\},
\end{alignat*}
and the set
\[
X:=\bigcup_{t\in [0,T]} \big[ X(t)\times \{t\} \big]\subset H^2(\R^n)\times [0,T] \subset H^1(\R^n)\times [0,T],   
\]
since the elements of $X(t)$ belong to $H^2(\R^n)$, being in $W^{2,\infty}(\R^n)$ and having compact support.

\begin{remark}\label{remarkvitale}
    From Proposition \ref{prop:useful}, item \emph{i)}, there exists $L_T>0$, depending only on $M,T,n,C_U$, such that for all $t\in [0,T)$, $\bar{m}\in X(t)$, and $\alpha\in\mathcal{A}$,
    \begin{equation*}
        \|m(\cdot,s_1;\alpha,\bar{m},t)-m(\cdot,s_2;\alpha,\bar m,t)\|_{H^1}
        \le L_T|s_1-s_2|\qquad \forall s_1,s_2\in [t,T].
    \end{equation*}
    Moreover, the constant $\tilde{L}$ cited in the item \emph{ii)} for the $L^2$ (or $H^1$) stability w.r.t. initial data depends only on the same constants, provided $m_1\in X(t_1)$ and $m_2\in X(t_2)$.
\end{remark}

\begin{proposition}\label{prop:X-invariant-compact}
    The set $X$ is invariant for the continuity evolution, in the sense that for every $(\bar m,t)\in X$ and $\alpha\in\mathcal{A}$ one has
    $(m(\cdot,s;\alpha,\bar m,t),s)\in X$ for all $s\in[t,T]$. Moreover, $X$ is compact in the product topology of $H^1(\R^n)\times[0,T]$.
\end{proposition}
\begin{proof}
    We prove compactness, since the invariance follows from the construction of $X$. \\ 
    Precompactness follows since the projection of $X$ onto the first component (which is clearly $X(T)$) is bounded in $H^2(B(0,R))$ for some $R>0$, and therefore relatively compact in
    $H^1(B(0,R))$. The extension by zero yields precompactness in $H^1(\R^n)$, while the time component ranges in $[0,T]$.\\
    Closedness is obtained by taking $(m_n,t_n)\to(\bar m,\bar t)$ in $H^1(\R^n)\times[0,T]$ and using the uniform $W^{2,\infty}$ bounds for $\{m_n\}$. These yield uniform $L^\infty$
    bounds and equi-Lipschitz continuity, hence (by Ascoli--Arzel\`a) a subsequence converges uniformly to some $\tilde m$. The $H^1$ convergence forces $\tilde m=\bar m$, so $m_n\to\bar m$
    uniformly and $\supp(\bar m)\subseteq \tilde{\Omega}(\bar t)$.\\
    Next, fix $s>\bar t$ and $\alpha\in\mathcal{A}$. Using the $H^1$ version of Proposition \ref{prop:useful}, item ii), one shows that
    \[
        m(\cdot,s;\alpha,m_n,t_n)\to m(\cdot,s;\alpha,\bar m,\bar t)\quad \text{in } H^1(\R^n).
    \]
    On the other hand, $\{m(\cdot,s;\alpha,m_n,t_n)\}_n$ is uniformly bounded in $W^{2,\infty}$, hence (up to a subsequence) converges weak-$\ast$ in $W^{2,\infty}$. The strong $H^1$
    convergence identifies the weak-$\ast$ limit with $m(\cdot,s;\alpha,\bar m,\bar t)$, and weak-$\ast$ lower semicontinuity yields for all $s>\bar t,\;\alpha \in\mathcal{A}$
    \[
        \|m(\cdot,s;\alpha,\bar m,\bar t)\|_{W^{2,\infty}}
        \le \liminf_{n\to\infty}\|m(\cdot,s;\alpha,m_n,t_n)\|_{W^{2,\infty}}
        \le \tilde C_U.
    \]
    Therefore $(\bar m,\bar t)\in X$, so $X$ is closed and hence compact.
\end{proof}

\begin{remark}\label{rem:mod-cont}
    We recall that (see \cite[Subsection 1.3.1]{Grafakos2014ModernFourierAnalysis}) the Sobolev space $H^s(\R^n)$ can be endowed with the equivalent norm
    \[
        \|w\|_{H^s_{\mathcal{F}}}:= \left[ \int_{\R^n} (1+|\xi|^2)^s|\hat{w}(\xi)|^2\;d\xi \right]^{\frac{1}{2}}\qquad \forall w\in H^s(\R^n),
    \]
    where $\hat{w}:\R^n\to \R$ denotes the Fourier transform of $w$.\\ 
    Now, for any $w \in H^2(\R^n)$, using Plancherel's formula,
    \begin{align*}
        \|w\|_{H^1_{\mathcal{F}}}^2
        & =  \int_{\R^n} (1+|\xi|^2)|\hat{w}(\xi)|^2\;d\xi
        \leq \|(1+|\cdot|^2)\hat{w}\|_{L^2}\,\|\hat{w}\|_{L^2} = \|w\|_{H^2_{\mathcal{F}}}\|w\|_{L^2}.
    \end{align*}
    Therefore there exists fixed constants $C_1,C_2>0$ such that for any $m_1,m_2\in X(T)$,
    \begin{align}
        \nonumber \|m_1-m_2\|_{H^1}
        &\leq C_1 \|m_1-m_2\|_{H^1_{\mathcal{F}}}
        \leq  C_1 \sqrt{\|m_1-m_2\|_{H^2_{\mathcal{F}}}}\sqrt{\|m_1-m_2\|_{L^2}} \\
        \label{eq:radice} &\leq C_1\sqrt{2C_2 \tilde{C}_U}
        \sqrt{\|m_1-m_2\|_{L^2}}
        =:C_X \sqrt{\|m_1-m_2\|_{L^2}}.
    \end{align}
    Note that \eqref{eq:radice} also implies that the $L^2$ and $H^1$ topologies are equivalent on $X(T)$.
\end{remark}

\section{The state-constrained optimal control problem}
\label{sec:presentation-problem}

In this section we present the state-constrained optimal control problem in detail and introduce the associated value function.
After some comments about the Dynamic Programming Principle we prove, under our hypotheses, the Lipschitz continuity of the value function in Subsection \ref{subsec:regularity}. In the following, the notation is the same as introduced in the previous sections.\\

\noindent We consider the running cost $\ell: X(T) \times \tilde{\mathcal{A}}\to [0,+\infty)$ and the final cost $\psi:X(T)\to [0,+\infty)$ which induce the Bolza-type cost functional $J:X\times \mathcal{A} \to [0,+\infty)$  
\[
    J(\bar{m},t,\alpha):= \int_{t}^{T} \ell(m(\cdot,s;\alpha,\bar{m},t),\alpha(\cdot,s))\;ds+\psi(m(\cdot,T;\alpha,\bar{m},t)).
\]
\begin{assumption}\label{ass:ellandpsi}
        We assume that $\ell$ and $\psi$ are bounded (by some constant $N_\ell,N_\psi>0$, respectively), continuous and that there exists $L_\ell,L_\psi>0$ such that
        \begin{enumerate}
            \item[i)] $\displaystyle  |\ell(m_1, a)-\ell (m_2, a)|\leq L_\ell\|m_1-m_2\|_{L^2}\qquad \forall m_1,m_2\in X(T),\;\forall a\in \tilde{\mathcal{A}},$
            \item[ii)] $\displaystyle |\psi(m_1)-\psi(m_2)|\leq L_\psi \|m_1-m_2\|_{L^2}\qquad \forall m_1,m_2\in X(T).$
        \end{enumerate}
\end{assumption}

\noindent Let $\Omega\subset \tilde{\Omega}$ open and $\delta>0$ be fixed. We are interested in minimizing the cost $J$ under a state constraint that enforces the mass, weighted by a fixed positive function $p:\R^n\to \R$, to be somehow concentrated in $\Omega$ at every time, namely
\begin{equation}\label{eq:constraint}
    \int_{\R^n\setminus\Omega} p(x)m(x)\;dx \leq \delta.
\end{equation}
 
\begin{assumption}\label{ass:Omegaandweight}
    For explicit computations, we assume $\Omega\subset \R^n$ is convex with boundary regular enough for the divergence theorem, and set $p(x):=\dist(x,\Omega)$; thus the constraint is $\langle p,m\rangle:=\langle p,m \rangle_{L^2(\tilde{\Omega}(T))}\leq \delta$. The same analysis applies to more general domains and weight functions (as $x\mapsto k\operatorname{dist}(x,\Omega)$ and more), see Remark \ref{rem:generalization domain and weight}).
\end{assumption}

\begin{notation}
    For $t\in[0,T]$, set
    $X_c(t):= \{ m\in X(t): \langle p,m\rangle<\delta \}$ and
    $\bar{X}_c(t):= \{ m\in X(t): \langle p,m\rangle\le\delta \}$,
    and define
    \[
        X_c:= \bigcup_{t\in [0,T]} X_c(t)\times \{t\},\qquad
        \bar{X}_c:= \bigcup_{t\in [0,T]} \bar{X}_c(t)\times \{t\}.
    \]
    Note that $\displaystyle \bar{X}_c=X\cap \left( \left\{ m\in H^1(\R^n) : \langle p,m\rangle\leq \delta \right\}\times[0,T] \right)$; then by Proposition \ref{prop:X-invariant-compact} it follows that $\bar{X}_c$ is compact in $H^1(\R^n)\times [0,T]$.
\end{notation}

 \noindent We are interested in minimizing the cost only over those controls in $\mathcal{A}$ that ensure the state constraint is never violated. Thus, for each $(\bar{m},t)\in \bar{X}_c$ we define
\[
    \mathcal{A}(\bar{m},t):=\left\{ \alpha \in \mathcal{A} : (m(\cdot,s;\alpha,\bar{m},t),s) \in \bar{X}_c \;\;\forall s\in[t,T] \right\},
\]
and note that at least the null constant control $\alpha\equiv 0$ always belongs to  $\mathcal{A}(\bar{m},t)$.\\  
The value function of the constrained problem is then
\[
    V:\bar{X}_c\to [0,+\infty)\qquad\qquad (\bar{m},t)\mapsto \inf_{\alpha\in \mathcal{A}(\bar{m},t)} J(\bar{m},t,\alpha).
\]

\noindent Relying on the semigroup property of the constrained dynamics and on the stability of the admissible class under concatenation of controls, we can follow the same argument as in the finite-dimensional Bolza setting. Therefore, see for example \cite[Chapter~III]{BardiCapuzzoDolcetta1997}, the value function $V$ satisfies the classical Dynamical Programming Principle stated below.

\begin{theorem}[DPP]\label{th:dpp}
    Given $(\bar{m},t)\in \bar{X}_c$, then for all $\tau \in [t,T]$
    \begin{equation}\label{eq:dpp}
         V(\bar{m},t) = \inf_{\alpha \in \mathcal{A}(\bar{m},t)} \left[ \int_t^{\tau} \ell(m(\cdot,s;\alpha,\bar{m},t),\alpha(\cdot,s))\;ds+V( m(\cdot,\tau;\alpha,\bar{m},t),\tau  ) \right].
    \end{equation}
\end{theorem} 

\begin{example}\label{example}
    \textcolor{red}{We give an explicit example of an interesting setting where the state constraint and the minimizing cost are somehow in conflict. Let $\Omega$ be the unit ball in $\mathbb{R}^n$, $\zeta$ a bell function such that $\operatorname{supp}(\zeta) = \bar\Omega$, $p(x):=k\operatorname{dist}(x,\Omega)$ with $k\gg1$, $\ell \equiv 0$ and $\psi(m):=\int_{\R^n} \zeta(x)m^2(x)\;dx$. Due to the particular feature of $\zeta$ and the high penalization just outside $\Omega$, it is expected that the final configuration of the mass will be concentrated around $\partial \Omega$. However, in different real-world applications, high concentration should be discouraged, and the role of $m^2$ in the cost, much more of having simply $\tilde \psi(m):=\int_{\R^n} \zeta(x)m(x)\;dx$, certainly goes in that direction. We end this example with the following important consideration (see also the Introduction): while $\tilde\psi$ can be extended to the Wasserstein setting in a natural way by continuity, the cost $\psi$ instead does not seem to have a natural counterpart in the measures' setting. This strengthens our idea that, at least in some cases, working in Hilbert setting with densities instead of measures may be preferable.}
\end{example}

\subsection{Value function regularity}\label{subsec:regularity}
In this subsection we prove that, under Assumptions \ref{ass:ellandpsi} and \ref{ass:Omegaandweight}, the value function is Lipschitz continuous on $\bar{X}_c$.
Following Soner's state-constraint strategy \cite{Soner1986StateConstraintI}, we combine DPP \eqref{eq:dpp} and a trajectories correction by a suitable pushback control near the boundary to obtain Lipschitz estimates for $V$.\\

\noindent In order to construct that suitable pushback control, we first need the following inequality. Fixed $x_\Omega \in \Omega$, we observe that
\begin{equation}\label{eq:uniform}
    \nabla \dist(x,\Omega) \cdot (x_\Omega-x) \leq -\dist(x_\Omega,\partial\Omega) =: -r_0<0\qquad \text{for a.e. }x\in \Omega^c.
\end{equation}
First of all note that $\nabla \dist(\cdot,\Omega)$ exists a.e. since $x\mapsto \dist(x,\Omega)$ is 1-Lipschitz; furthermore, let $x\in \Omega^c$ be a differentiability point and let $y\in\partial\Omega$ be its (unique) projection onto $\Omega$. By convexity there exists a supporting hyperplane with outward unit normal $n(y)$ such that
\begin{equation*}
    n(y)\cdot (z-y) \leq 0 \qquad \forall z\in \Omega.
\end{equation*}
Let $\delta \in (0,r_0)$. Using $ z_\delta:=x_\Omega + (r_0-\delta)n(y)\in \Omega$ above we obtain
\[
    0\geq n(y) \cdot(z_\delta-y) = n(y) \cdot (x_\Omega-y) + (r_0-\delta ) \;\; \forall \delta \in (0,r_0)\Rightarrow n(y) \cdot (x_\Omega-y) \leq -r_0.  
\] 
Since $x = y + \dist(x,\Omega) n(y)$ and $n(y)=\nabla \dist(x,\Omega)$, we obtain \eqref{eq:uniform}; indeed,
\[ 
    \nabla \dist(x,\Omega) \cdot (x_\Omega-x) =  n(y) \cdot (x_\Omega-y) - \dist(x,\Omega) \leq -r_0.
\]
\textcolor{red}{Recall that $\tilde{\Omega}(2T)=\{x\in\R^n:\dist(x,\tilde{\Omega})\leq 2MT\}$.} Now fix $\eps\ll1$ and let $\phi\in C^\infty(\R^n)$ be a cutoff function with $0\leq\phi\leq1$ such that
\begin{equation*}\label{cutoff_def}
    \phi(x)=
    \begin{cases}
    1 & \text{if }x\in {\color{red}\tilde{\Omega}(2T)},\\
    0 & \text{if }x\in \{x\in\R^n:\operatorname{dist}(x,{\color{red}\tilde{\Omega}(2T)})\ge \eps\},
    \end{cases}
\end{equation*}
so that $x\mapsto \phi(x) (x_\Omega-x) \in W^{3,\infty}(\R^n)$. After rescaling we define the constant (in time) \emph{pushback control}  
\begin{equation}\label{eq:pushback-control-def}
    a_p:\R^n\to \R^n \qquad x\mapsto a_p(x):=\bar{M}\phi(x)(x_\Omega-x),
\end{equation}
where $\bar{M}\ll1$ is such that $a_p\in\tilde{\mathcal{A}}$; by construction we have, by \eqref{eq:uniform}, (recall that $p(x) = \dist(x,\Omega) $)
\begin{equation}\label{eq:pushback inequality}
    \nabla p(x) \cdot a_p(x) \leq -\bar{M} r_0 < 0\qquad \text{for a.e. }x\in {\color{red}\tilde{\Omega}(2T)}\setminus \Omega.
\end{equation}

\begin{notation}
    Let $\eta: L^2(\tilde{\Omega}(T))\supset X(T) \to \R$ be defined by $m\mapsto \eta(m):=\displaystyle \delta - \langle p,m \rangle $. Clearly for any $m\in X(T)$ we have $m\in \bar{X}_c(T)$ if and only if $\eta(m)\geq 0$. \textcolor{red}{Also note that $t\mapsto \langle p,m(t) \rangle$ is nonincreasing along the trajectories $t\mapsto m(t)$ generated by the pushback control \eqref{eq:pushback-control-def}, as a sort of Lyapunov function.}
\end{notation}

\begin{lemma}\label{lemma:remake soner}
        There exist $t^*,C^*>0$ such that for all $(m_0,t_0)\in\bar{X}_c$ and $\alpha\in\mathcal{A}$, there exists $\bar{\alpha}\in \mathcal{A}(m_0,t_0)$ for which the following holds:
        \begin{equation}\label{da dim}
            \left| J_{h}(m_0,t_0,\bar\alpha)-J_{h}(m_0,t_0,{\alpha}) \right| \leq C^* \max \left\{0,\sup_{t_0\leq s\leq t_0+t^*} \big[-\eta(m(\cdot,s;\alpha,m_0,t_0))\big] \right\}=:C^*\eps
        \end{equation}
        for all $h\in {\color{red}(0,\min\{t^*,T-t_0\}]}$, where $\displaystyle J_{h}(m_0,t_0,\alpha):=\int_{t_0}^{t_0+h}\ell(m(\cdot,t;\alpha,m_0,t_0),\alpha(\cdot,t))\;dt$. In \eqref{da dim}, we extend $\alpha$ to $[t_0,+\infty)\times \R^n$, if necessary, by setting $\alpha(\cdot,t)\equiv 0$ for all $t>T$.
\end{lemma}

\begin{proof}
    Let $t^*,k>0$ be fixed (we will specify them later), and let $(m_0,t_0)\in \bar{X}_c$, $\alpha\in\mathcal{A}$; set $m(\cdot,\cdot):=m(\cdot,\cdot;\alpha,m_0,t_0)$. We define $\displaystyle t_e:=\inf\{t\in[t_0,T]:\eta(m(\cdot,t))<0\}-t_0$, which represents the amount of time before the constraint is possibly violated under the control $\alpha$ (with the standard assumption $\inf\emptyset = +\infty$). Note that both $t_e$ and $\eps$ depend on the triple $(m_0,t_0,\alpha)$.\\
    We now define a control $\bar{\alpha}$ on the interval $[t_0,t_0+t^*]$ as follows. Set
    \[
        t_1:=\min\{t_e,t^*\},\qquad t_2:=\min\{t^*,\,t_e+k\eps\},
    \]
    and define
    \[
        \bar{\alpha}(\cdot,t):=
        \chi_{[t_0,t_0+t_1)}(t)\alpha(\cdot,t)
        +\chi_{[t_0+t_1,\,t_0+t_2]}(t)a_p
        +\chi_{(t_0+t_2,\,t_0+t^*]}(t)\alpha(\cdot,t-k\eps).
    \]
    where $a_p\in\tilde{\mathcal{A}}$ is the pushback control \eqref{eq:pushback-control-def}; we denote $m(\cdot,\cdot;\bar\alpha,m_0,t_0)$ by $\bar{m}(\cdot,\cdot)$. We claim that $\eta(\bar{m}(\cdot,t))\geq 0$ for all $t\in[t_0,t_0+t^*]$.\\

    \noindent If $t_e>t^*$, then $\bar{\alpha}=\alpha$ on $[t_0,t_0+t^*]$ and the claim is immediate.\\ 
    We now treat the case $t_e\leq t^*\leq t_e+k\eps$. Clearly $\eta(\bar{m}(\cdot,t))\geq 0$ for every $t\in [t_0,t_0+t_e]$. \\
    For almost every $t\in (t_0+t_e,t_0+t^*)$ we compute, by \eqref{eq:continuity},the divergence theorem and the definition of $\eta$,
    \begin{align}
        \label{eq:sopra}\frac{d}{dt}\eta(\bar{m}(\cdot,t)) 
        & = -\int_{\R^n}p(x)\partial_t\bar{m}(x,t)\;dx = \int_{\R^n}p(x)\diver(\bar\alpha(x,t)\bar{m}(x,t))\;dx\\
        \nonumber & = \int_{\tilde{\Omega}(T)\setminus\Omega}p(x)\diver(\bar\alpha(x,t)\bar{m}(x,t))\;dx =
        - \int_{\tilde{\Omega}(T)\setminus\Omega} \big( \nabla p(x)\cdot \bar\alpha(x,t) \big)\bar{m}(x,t)\;dx,
    \end{align}
    Using \eqref{eq:pushback inequality}, we obtain $\displaystyle \frac{d}{dt}\eta(\bar{m}(\cdot,t)) \geq \bar{M}r_0 \int_{\tilde{\Omega}(T)\setminus\Omega} \bar{m}(x,t)\;dx$.\\ 
    Now, it follows straightforwardly from Proposition \ref{prop:useful} (and in particular from Remark \ref{remarkvitale}) that $t\mapsto \eta(\bar{m}(\cdot,t))$ is uniformly Lipschitz with respect to $(m_0,t_0,\alpha)$; we denote its Lipschitz constant by $L_\eta$.\\
    Set $\bar{t}:=\delta/(2L_\eta)$. Then, for every $t\in (t_0+t_e,t_0+\bar{t})$, we have
    \begin{equation}\label{eq:sopra2}
        \eta(\bar{m}(\cdot,t)) \leq \eta(\bar{m}(\cdot,t_0+t_e)) + L_\eta(t-t_0-t_e) \leq  \frac{\delta}{2}
        \Rightarrow \int_{\R^n} p(x)\bar{m}(x,t)\;dx \geq \frac{\delta}{2}.
    \end{equation}
    Therefore, for every $t\in (t_0+t_e,t_0+\bar{t})$, we get
    $ \displaystyle \int_{\tilde{\Omega}(T)\setminus \Omega} \bar{m}(x,t)\;dx \geq \frac{\delta}{2\|p\|_{L^\infty(\tilde{\Omega}(T))}}$.\\ 
    Therefore, requiring (see below) that $t^*\leq \bar{t}$, we have
    \begin{equation}\label{eq:ferrara}
        \frac{d}{dt}\eta(\bar{m}(\cdot,t)) \geq \frac{\delta\bar{M}r_0}{2\|p\|_{L^\infty(\tilde{\Omega}(T))}} =:C_1 > 0\qquad \text{for a.e. }t\in (t_0+t_e,t_0+t^*),
    \end{equation}
    which, together with $\eta(\bar{m}(\cdot,t_0+t_e)) = 0$, guarantees that
    $\eta(\bar{m}(\cdot,t)) \geq 0$ for every $t\in [t_0,t_0+t^*]$ also in this case.\\
    We now focus on the case $t_e+k\eps < t^*$. Clearly, $\eta(\bar{m}(\cdot,t))\geq 0$ for all $t\in [t_0,t_0+t_e+k\eps]$. Let $t\in(t_0+t_e+k\eps,t_0+t^*]$. Then, for every $s\in (t_0+t_e+k\eps,t]$, by adding and subtracting to the first line of \eqref{eq:sopra} the term
    $\displaystyle \int_{\R^n} p(x)\diver(\bar{\alpha}(x,s)m(x,s-k\eps))\,dx$ and using H\"{o}lder's inequality together with the uniform bounds on $\mathcal{A}$ we obtain, for a suitable $M'>0$,
    \begin{align}
        \nonumber \frac{d}{ds}\eta(\bar{m}(\cdot,s)) \geq & -M'\|p\|_{L^\infty(\tilde{\Omega}(T))}\|\bar{m}(\cdot,s)-m(\cdot,s-k\eps)\|_{H^1(\R^n)}\\  
        \label{eq:troncato}& + \int_{\R^n} p(x)\diver(\alpha(x,s-k\eps)m(x,s-k\eps))\;dx,
    \end{align}
    where in the last line we used $\bar{\alpha}(\cdot,s)=\alpha(\cdot,s-k\eps)$ for $s\in (t_0+t_e+k\eps,t]$. Furthermore we observe that
    $u(\cdot,s):=\bar{m}(\cdot,s)-m(\cdot,s-k\eps)$ solves the continuity equation. Hence, using first Proposition \ref{prop:energy-ex} and then Remark \ref{remarkvitale} (and noting that $\bar{m}(\cdot,t_0+t_e)=m(\cdot,t_0+t_e)$), we get
    \begin{align}
        \label{primopez} \|\bar{m}(\cdot,s)-m(\cdot,s-k\eps)\|_{H^1(\R^n)} \leq e^{C(s-t_0-t_e-k\eps)}L_Tk\eps \qquad \forall s\in (t_0+t_e+k\eps,t].
    \end{align}
    Moreover, recalling that $\displaystyle\int_{\R^n} p(x)m(x,t_0+t_e)\;dx=\delta$ and the definition of $\eps $ in \eqref{da dim},
    \begin{align*}
        \int_{t_0+t_e+k\eps}^t \int_{\R^n} p(x)&\diver(\alpha(x,s-k\eps)m(x,s-k\eps))\;dx\;ds  = - \int_{\R^n}\int_{t_0+t_e+k\eps}^t p(x)\partial_t m(x,s-k\eps)\;ds\;dx\\ 
        & = -\int_{\R^n} p(x)\left[ m(x,t-k\eps)-m(x,t_0+t_e) \right]\;dx = \eta(m(\cdot,t-k\eps))\\ 
        & \geq \inf_{t_0+t_e\leq s\leq t_0+t^*-k\eps} \left[ \eta(m(\cdot,s))\right] \geq -\sup_{t_0\leq s\leq t_0+t^*}  \left[ -\eta(m(\cdot,s))\right] = -\eps.
    \end{align*} 
    Using this together with \eqref{eq:ferrara}, \eqref{eq:troncato}, and \eqref{primopez}, it follows ($C_2:=M'\|p\|_{L^\infty(\tilde{\Omega}(T))}L_T/C$) that
    \begin{align*}
        \eta(\bar{m}(\cdot,t)) & = \eta(\bar{m}(\cdot,t_0+t_e) + \int_{t_0+t_e}^{t_0+t_e+k\eps} \frac{d}{ds}\eta(\bar{m}(\cdot,s))\;ds + \int_{t_0+t_e+k\eps}^{t} \frac{d}{ds}\eta(\bar{m}(\cdot,s))\;ds \\ 
        & \geq 0+C_1k\eps -C_2k\eps  \int_{t_0+t_e+k\eps}^t Ce^{C(s-t_0-t_e-k\eps)}\;ds -\eps  \geq C_1k\eps - C_2k\eps \left[ e^{Ct^*}-1 \right]-\eps.
    \end{align*}
    Therefore, taking
    \[  
         t^*:=\min \left\{ \frac{\ln \left( \frac{C_1}{2C_2}+1 \right)}{C}, \bar{t} \right\}\qquad k:=\frac{2}{C_1} ,
    \]
    a direct computation shows that for every $t\in(t_0+t_e+k\eps,t_0+t^*]$ we have
    $\eta(\bar{m}(\cdot,t)) \geq 0$, so this last case is also settled.\\
    Therefore, the claim is proved. By iterating this construction for successive time intervals of length $t^*$ (up to the final time horizon $T$), we obtain $\bar\alpha\in \mathcal{A}(m_0,t_0)$. \\ 
    Let $h \in (0,t^*]$. If $h \in (0,t_e]$, then \eqref{da dim} is immediate; if instead $h \in (t_e, t_e + k\eps)$, then
    \begin{align*}
        |J_{h}(m_0,\bar\alpha)-J_{h}(m_0,\alpha)| & \leq \int_{t_0+t_e}^{t_0+h}\big|\ell (\bar{m}(\cdot,t),\bar\alpha(\cdot,t)) - \ell ({m}(\cdot,t),\alpha(\cdot,t))\big|\;dt \leq 2N_\ell k\eps.
    \end{align*}
    If instead $h \in [t_e+k\eps,t^*]$, then (hereafter, $m(t)=m(\cdot,t)$ and $\alpha(\cdot,t)=\alpha(t)$)
        \begin{align*}
        \nonumber \left|J_{h}(m_0,\bar\alpha)-J_{h}(m_0,\alpha)\right| &  = \left|\int_{t_0+t_e}^{t_0+t_e+k\eps} \ell(\bar m(t),\bar\alpha(t))\;dt + \int_{t_0+t_e+k\eps}^{t_0+h} \ell(\bar m(t),\bar\alpha(t))\;dt \right.\\
        \nonumber & \pm \int_{t_0+t_e+k\eps}^{t_0+h} \ell( m(t-k\eps),\underbrace{\alpha(t-k\eps)}_{=\bar\alpha(t)})\;dt-\int_{t_0+t_e}^{t_0+h} \ell( m(t),\alpha(t))\;dt\Bigg|\\
        & \leq \int_{t_0+t_e}^{t_0+t_e+k\eps} |\ell(\bar m(t),\bar\alpha(t))|\;dt + L_\ell\int_{t_0+t_e+k\eps}^{t_0+h} \|\bar{m}(t)-m(t-k\eps)\|_{H^1(\R^n)}\;dt\\ 
        & + \left|\int_{t_0+t_e}^{t_0+h-k\eps} \ell(m(t),\alpha(t))\;dt - \int_{t_0+t_e}^{t_0+h} \ell(m(t),\alpha(t))\;dt\right|\\ 
        & \leq 2N_\ell k\eps + L_\ell L_T k\eps \int_{t_0+t_e+k\eps}^{t_0+h} e^{C(t-t_0-t_e-k\eps)}\;dt\leq C^* \eps,  
    \end{align*}
    for a suitable uniform constant $C^*>0$, where \eqref{primopez} was used in the last line.
    \end{proof}

\begin{remark}\label{rem:generalization domain and weight}
\textcolor{red}{Lemma \ref{lemma:remake soner} also holds for nonconvex domains $\Omega$ and nonnegative weights $p\in C^0(\mathbb R^n)$, with $p|_{\overline\Omega}=0$, under the following sufficient condition. Set $K:=\overline{\widetilde\Omega(2T)\setminus\Omega}$ and assume that $p\in C^1(V\setminus\overline\Omega)$ for an open neighborhood $V\supset K$, with $\nabla p$ extending continuously to $g:K\to\mathbb R^n$. Set $\underline\gamma:=\min_K|g|>0$ and $\overline\gamma:=\max_K|g|$.
By continuous extension of $-g$ to $\mathbb R^n$, an appropriate cutoff and standard mollification, we can choose $b\in C_c^\infty(\mathbb R^n;\mathbb R^n)$ with $\|b+g\|_{C(K)}<\underline\gamma^2/(2\overline\gamma)$.
For sufficiently small $\overline M>0$, the field $a_p:=\overline M b$ belongs to $\widetilde A$ and satisfies, for a.e. $x\in\widetilde\Omega(2T)\setminus\Omega$,
\[
\nabla p(x)\cdot a_p(x)
=\overline M\bigl(-|g(x)|^2+g(x)\cdot(b(x)+g(x))\bigr)
\le-\overline M\underline\gamma^2
+\overline M\overline\gamma\frac{\underline\gamma^2}{2\overline\gamma}
=-\frac{\overline M\underline\gamma^2}{2}<0,
\]
which gives the required pushback property \eqref{eq:pushback inequality}.}
\end{remark}

\begin{theorem}\label{th:lipcontinuity-of-V}
    The value function $V$ is Lipschitz continuous on $\bar{X}_c$: there exists $L_V>0$ such that
    \[
        |V(m_1,t_1)-V(m_2,t_2)|
        \leq L_V\!\left(\|m_1-m_2\|_{L^2(\R^n)}+|t_1-t_2|\right)
        \qquad \forall (m_1,t_1),(m_2,t_2)\in \bar{X}_c.
    \]
\end{theorem}

\begin{proof}
    For each $t\in [0,T]$ we define, for $r>0$,
    \[
        \omega_t(r):=\sup \left\{ |V(m_1,t)-V(m_2,t)|:\ (m_j,t)\in \bar{X}_c \ \text{for } j=1,2 \ \text{and}\ \|m_1-m_2\|_{L^2(\R^n)} \leq r \right\}.
    \]
     Let $t^*$ be as in Lemma \ref{lemma:remake soner}. Fix $r>0$ and consider $(m_1,t),(m_2,t)\in \bar{X}_c$ such that $\|m_1-m_2\|_{L^2}\leq r$, $\beta>0$, $h\in (0,\min\{T-t,t^*\}]$, and let $\alpha\in \mathcal{A}(m_2,t)$ be $\beta$-optimal for $(m_2,t)$ in the DPP of Theorem \ref{th:dpp}, i.e.,
     \begin{equation}\label{eq:lipcont-1}
         V(m_2,t) \geq J_h(m_2,t,\alpha) + V(m(\cdot,t+h;\alpha,m_2,t),t+h)-\beta.
     \end{equation}
    Let $\bar{\alpha}\in \mathcal{A}(m_1,t)$ be associated with $\alpha$ as in Lemma \ref{lemma:remake soner}. Note that for every $s\in [t,t+t^*]$, by adding and subtracting $\langle p,m(\cdot,s;\alpha,m_2,t) \rangle$, since $-\eta(m(\cdot,s;\alpha,m_2,t))\leq 0$ and recalling Remark \ref{remarkvitale},
    \begin{align*}
        \nonumber -\eta(m(\cdot,s;\alpha,m_1,t)) & = -\eta(m(\cdot,s;\alpha,m_2,t)) + \int_{\R^n} p(x)\big[ m(x,s;\alpha,m_1,t)-m(x,s;\alpha,m_2,t) \big]\;dx\\
        \nonumber & \leq \|p\|_{L^2(\tilde{\Omega}(T))}\tilde{L}\|m_1-m_2\|_{L^2(\R^n)}\leq \|p\|_{L^2(\tilde{\Omega}(T))}\tilde{L}r .
    \end{align*}  
    Therefore, it follows from \eqref{da dim} (in particular from the definition of $\eps$) that
    $\displaystyle\left| J_h(m_1,t,\alpha)-J_h(m_1,t,\bar\alpha)  \right|\leq C^*\|p\|_{L^2(\tilde{\Omega}(T))}\tilde{L}r=:C_0r$.\\ 
    Moreover, Assumption \ref{ass:ellandpsi} implies that
    \[  
         \left| J_h(m_1,t,\alpha)-J_h(m_2,t,\alpha)  \right| \leq L_\ell \tilde{L}\int_t^{t+h}\|m_1-m_2\|_{L^2(\R^n)}\;ds \leq C_1 r,
    \]
    with $C_1:=L_\ell \tilde{L}T$; therefore
    \begin{equation*}
        \left| J_h(m_1,t,\bar{\alpha})-J_h(m_2,t,\alpha)\right| \leq (C_0+C_1)r=:C_2r.
    \end{equation*}
    Hence, using this and \eqref{eq:lipcont-1}, we obtain
    \begin{align}
        \label{eq:lipcont-2} V(m_1,t)-V(m_2,t)
         & \leq  J_h(m_1,t,\bar\alpha)+V(m(\cdot,t+h;\bar\alpha,m_1,t),t+h)   -V(m_2,t)\\
        \nonumber & \leq  C_2r +  V(m(\cdot,t+h;\bar\alpha,m_1,t),t+h)- V(m(\cdot,t+h;\alpha,m_2,t),t+h)  + \beta.
    \end{align}
    By construction $\bar\alpha$ is as $\alpha$ up to the first possible exit time for $(m_1,t)$, actually $t+t_e$ (see the proof of Lemma \ref{lemma:remake soner}), then it substantially differs from $\alpha$ only on a time interval of length $k\eps$, being afterwards only a time-shift of $\alpha$. Then using uniform bounds (see Remark \ref{remarkvitale}) and \eqref{primopez}, we can find $C_3>0$, independent of $(m_1,t)$ and $\alpha$, such that
    \[
        \|m(\cdot,s;\bar\alpha,m_1,t)-m(\cdot,s;\alpha,m_1,t)\|_{L^2(\R^n)}
        \leq C_3\eps \qquad \forall s\in [t,t+t^*],
    \]
    and since $\displaystyle \eps\leq \|p\|_{L^2(\tilde{\Omega}(T))}\tilde{L}\,r$, we get $\displaystyle  \|m(\cdot,s;\bar\alpha,m_1,t)-m(\cdot,s;\alpha,m_1,t)\|_{L^2(\R^n)}\leq C_4r$ (for some $C_4>0$) for all $s\in [t,t+t^*]$. In particular for $s=t+h$, by adding and subtracting $m(\cdot,t+h;\alpha,m_1,t)$ here below, we get
     \begin{align*}
        \|m(\cdot,t+h;\bar\alpha,m_1,t)-m(\cdot,t+h;\alpha,m_2,t)\|_{L^2}
         \leq C_4 r + \tilde{L} r =: \tilde{C}r.
    \end{align*}
    Then, \eqref{eq:lipcont-2}, by the definition of $\omega_{t+h}$, implies that $ V(m_1,t)-V(m_2,t) \leq C_2r+ \omega_{t+h}(\tilde{C}r)+\beta.$\\
    By arbitrariness of $\beta$, exchanging the roles of $m_1$ and $m_2$ and taking the supremum, we obtain
    \begin{equation}\label{eq:lipcont-3}
        \omega_t(r)\leq C_2 r + \omega_{t+h}(\tilde{C}r)
        \qquad  \forall t\in [0,T],\ h \in [0,\min\{ t^*,T-t\}],\ r>0.
    \end{equation}
    Set $t_j:=T-jt^*$, $\omega^j:=\omega_{t_j}$ and $\bar{j}:=\max\{j:\,t_j\ge0\}$. From \eqref{eq:lipcont-3} we have $\omega^{j+1}(r)\leq C_2r+\omega^j(\tilde C r)$ for $j=0,\dots,\bar{j}-1$; since, by Assumption \ref{ass:ellandpsi}, $\omega^0(r)=\omega_T(r)\le L_\psi r$, iteration yields $L_*>0$ such that $\omega^j(r)\le L_*r$ for all $j=0,\dots,\bar j$. \\ 
    If $t\in[0,T]\setminus\{t_0,\dots,t_{\bar j}\}$ and $\tilde j:=\max\{j:\,t_j>t\}$, choosing $h:=t_{\tilde j}-t\in(0,t^*)$ in \eqref{eq:lipcont-3} yields $\omega_t(r)\le C_2r+\omega_{\tilde j}(\tilde C r)\le (C_2+L_*\tilde C)r$.\\ 
    Therefore there exists $L_{\mathrm{mass}}>0$ such that $\omega_t(r)\le L_{\mathrm{mass}}r$ for all $t\in[0,T],\ r>0,$ i.e.,
    \begin{equation}\label{eq:lipcont-4}
        |V(m_1,t)-V(m_2,t)|
        \le L_{\mathrm{mass}}\|m_1-m_2\|_{L^2}
        \qquad \forall (m_1,t),(m_2,t)\in\bar{X}_c.
    \end{equation}
    
    Now fix $(m_0,t)\in\bar{X}_c$ and $h\in[0,T-t]$, and hence it is $(m_0,t+h)\in\bar{X}_c$. Applying the DPP (Theorem \ref{th:dpp}) for $(m_0,t)$ with $\tau = t + h$, adding and subtracting $V(m_0,t+h)$ and using Assumption \ref{ass:ellandpsi} and \eqref{eq:lipcont-4}, we obtain for every $\alpha\in\mathcal{A}(m_0,t)$
    \[
        V(m_0,t)\le N_\ell h + L_{\mathrm{mass}}\|m(\cdot,t+h;\alpha,m_0,t)-m_0\|_{L^2}+V(m_0,t+h).
    \]
    Hence, by Remark \ref{remarkvitale}, $V(m_0,t)-V(m_0,t+h)\le (N_\ell+L_{\mathrm{mass}}L_T)h=:L_{\mathrm{time}}h.$\\ 
    On the other hand, let $\beta>0$ and let $\alpha\in\mathcal{A}(m_0,t)$ be $\beta$-optimal for $(m_0,t)$ in the DPP with $\tau = t + h$; then
    \begin{align*}
         V(m_0,t)& \ge \int_t^{t+h}\ell(m(\cdot,s;\alpha,m_0,t),\alpha(\cdot,s))\,ds + V(m(\cdot,t+h;\alpha,m_0,t),t+h)-\beta\\ 
        & \ge V(m(\cdot,t+h;\alpha,m_0,t),t+h)-\beta,
    \end{align*}
    Therefore, by \eqref{eq:lipcont-4} and Remark \ref{remarkvitale},
    \[
        V(m_0,t+h)-V(m_0,t)
        \le |V(m_0,t+h)-V(m(\cdot,t+h;\alpha,m_0,t),t+h)|+\beta
        \le L_{\mathrm{mass}}L_T h+\beta.
    \]
    By the arbitrariness of $\beta$ (and combining with the inequality above), we obtain
    \begin{equation}\label{eq:lipcont-time-2}
        |V(m_0,t)-V(m_0,t+h)|\le L_{\mathrm{time}}h
        \qquad \forall (m_0,t)\in\bar{X}_c,\ \forall h\in[0,T-t].
    \end{equation}
    Finally, for arbitrary $(m_1,t_1),(m_2,t_2)\in\bar{X}_c$ (say $t_1\le t_2$), adding and subtracting $V(m_1,t_2)$ (recalling that $(m_1,t_2)\in\bar{X}_c$), the claim easily follows from \eqref{eq:lipcont-4} and \eqref{eq:lipcont-time-2}.
\end{proof}

\section{The constrained Hamilton--Jacobi--Bellman equation}\label{sec:hjb}
\noindent In this section, following the finite-dimensional Soner’s state-constraint approach \cite{Soner1986StateConstraintI}, we introduce the constrained Hamilton-Jacobi-Bellman equation (HJB), coupled with the terminal condition (TC) induced by the terminal cost $\psi$, and the notion of constrained viscosity solution. We show that the value function $V$ is a constrained viscosity solution of (HJB)+(TC).

\begin{remark}\label{rem:frechet}
    As noted in \cite[Remark 2.3]{bagagiolo2024single}, under our standing assumptions, for every $\varphi\in C^1\big(L^2(\R^n)\times[0,T]\big)$ and every solution $m$ of \eqref{eq:continuity}, the map $t\mapsto\varphi(m(\cdot,t),t)$ is absolutely continuous on $[0,T]$. Moreover, for a.e. $t\in(0,T)$,
    \begin{align*}
        \frac{d}{dt}\varphi(m(\cdot,t),t)
        & =\left\langle \nabla_m\varphi(m(\cdot,t),t),\partial_tm(\cdot,t)\right\rangle+\partial_t\varphi(m(\cdot,t),t)\\
        & =-\left\langle \nabla_m\varphi(m(\cdot,t),t),\diver\big(\alpha(\cdot,t)m(\cdot,t)\big)\right\rangle
          +\partial_t\varphi(m(\cdot,t),t),
    \end{align*}
    where $\nabla_m$ denotes the ($L^2$-)Fr\'echet gradient with respect to the state variable $m$.
\end{remark}

\noindent We define the Hamiltonian of our problem as 
    \begin{equation*}
        H:H^1(\R^n)\times L^2(\R^n) \to \R,\qquad (m,q)\mapsto H(m,q):= \sup_{a\in\tilde{\mathcal{A}}} \bigg[  \langle q,\diver(am)\rangle - \ell(m,a) \bigg].
    \end{equation*}
\begin{lemma}\label{lemma:semilip hamilton}
   For every $m_1,m_2\in X(T)$ and $q_1,q_2\in L^2(\R^n)$ we have
    \begin{align*}
         |H(m_1,q_1)-H(m_2,q_2)|
         \leq M\tilde{C}_U\|q_1-q_2\|_{L^2}+M\|q_1\|_{L^2}\|m_1-m_2\|_{H^1}+ L_\ell\|m_1-m_2\|_{L^2}.
    \end{align*}
\end{lemma}

\begin{proof}
    By a direct computation,
    \begin{align*}
        |H(m_1,q_1)-H(m_2,q_1)| & \leq \left|\sup_{a\in\tilde{\mathcal{A}}} \big[ \langle q_1,\diver(am_1)-\diver(am_2)\rangle+\ell(m_2,a)-\ell(m_1,a) \big]\right|\\
        & \leq \|q_1\|_{L^2}M\|m_1-m_2\|_{H^1}+L_\ell \|m_1-m_2\|_{L^2}.
    \end{align*}
    Moreover, $\displaystyle |H(m_2,q_1)-H(m_2,q_2)| \leq \left|  \sup_{a\in\tilde{\mathcal{A}}} \big[ \langle q_1-q_2,\diver(am_2)\rangle\big] \right|\leq M\tilde{C}_U \|q_1-q_2\|_{L^2}$.
\end{proof}

\noindent The Hamilton--Jacobi--Bellman equation, coupled with its terminal condition, is given by
    \begin{equation}\label{eq:HJB}
        \left\{
        \begin{alignedat}{2}
            -\partial_tu(m,t) + H(m,\nabla_m u(m,t)) &= 0,
            &\qquad& \text{(HJB)}\\
            u(m,T) &= \psi(m).
            &\qquad& \text{(TC)}
        \end{alignedat}
        \right.
    \end{equation}
    In the spirit of Soner \cite{Soner1986StateConstraintI}, we define constrained viscosity solutions of \eqref{eq:HJB} by requiring the subsolution condition in the interior and the supersolution condition up to the boundary; more precisely, we use the following definition.
\begin{definition}\label{def:solvisc}
   Let $u\in C(\bar{X}_c)$ (continuous with respect to $H^1$ as well as to $L^2$, see Remark \ref{rem:mod-cont}).
    \begin{itemize}
        \item[i)] $u$ is called a \emph{viscosity subsolution} of \eqref{eq:HJB} in $X_c$ if:
        \begin{itemize}
            \item[i-a)]\label{ia} for every $\varphi\in C^1(L^2(\R^n)\times[0,T])$ and every $\displaystyle (m_0,t_0)\in \bigcup_{t\in(0,T)} \left[ X_c(t)\times \{t\} \right]$ which is a local maximum point of $u-\varphi$ with respect to $X_c$, we have 
        \[ -\partial_t \varphi(m_0,t_0) + H(m_0,\nabla_m \varphi(m_0,t_0)) \leq 0; \]
            \item[i-b)]\label{ib} $u(m,T)\leq \psi(m)$ for all $m\in \bar{X}_c(T)$.
        \end{itemize}
        \item[ii)] $u$ is called a \emph{viscosity supersolution} of \eqref{eq:HJB} in $\bar{X}_c$ if:
        \begin{itemize}
            \item[ii-a)]\label{iia} for every $\varphi\in C^1(L^2(\R^n)\times[0,T])$ and every $\displaystyle (m_0,t_0)\in \bigcup_{t\in(0,T)} \left[ \bar{X}_c(t)\times \{t\} \right]$ which is a local minimum point of $u-\varphi$ with respect to $\bar{X}_c$, we have 
        \[ -\partial_t \varphi(m_0,t_0) + H(m_0,\nabla_m \varphi(m_0,t_0)) \geq 0; \]
            \item[ii-b)]\label{iib} $u(m,T)\geq \psi(m)$ for all $m\in \bar{X}_c(T)$.
        \end{itemize}
        \item[iii)] $u$ is called a \emph{constrained viscosity solution} of \eqref{eq:HJB} in $\bar{X}_c$ if it is both a viscosity subsolution in $X_c$ and a viscosity supersolution in $\bar{X}_c$.
    \end{itemize}
\end{definition}

\begin{remark}\label{rem:short-time}
    For every
    $(m_0,t_0)\in \bigcup_{t\in [0,T)} \left[ X_c(t)\times \{t\} \right]$, there exists $h=h(m_0,t_0)>0$ such that, for every $\alpha\in\mathcal{A}$, one has that $m(\cdot,t;\alpha,m_0,t_0)\in X_c(t)$ for all $ t\in[t_0,t_0+h]$. This follows immediately from the uniform (w.r.t. $(\alpha,m_0,t_0)$) Lipschitz continuity of $t\mapsto \langle p,m(\cdot,t;\alpha, m_0,t_0) \rangle$ (see Remark \ref{remarkvitale}).
\end{remark}

\begin{proposition}\label{prop:solvisc}
    The value function $V$ is a constrained viscosity solution of \eqref{eq:HJB} in $\bar{X}_c$.
\end{proposition}

\begin{proof}
    Note that $V(m,T)=\psi(m)$ for every $m\in \bar{X}_c(T)$ (hence \emph{i-b)} and \emph{ii-b)} hold). Moreover, $V\in C(\bar{X}_c)$ by Theorem~\ref{th:lipcontinuity-of-V}.\\
    \emph{i-a)} Let $\varphi\in C^1(L^2(\R^n)\times [0,T])$ and $(m_0,t_0)\in \bigcup_{t\in (0,T)} \left[ X_c(t)\times\{t\} \right]$ be such that $V-\varphi$ attains a local maximum relative to $X_c$ at $(m_0,t_0)$, i.e., there exists $r>0$ such that for all $(m,t)\in X_c$ with $\|m-m_0\|_{L^2}<r$ and $ |t-t_0|<r$
    \begin{equation}\label{eq:subsol1}
        V(m_0,t_0) - \varphi(m_0,t_0) \geq V(m,t)-\varphi(m,t).
    \end{equation}
    Let $h=h(m_0,t_0)$ be as in Remark~\ref{rem:short-time} and fix a (time-)constant control $a\in\tilde{\mathcal{A}}$. Set
    $\displaystyle \bar{h}:=\min\left(h,r/L_T,r\right)$. Then, for every $t\in [t_0,t_0+\bar{h}]$, one has $(m(\cdot,t;a,m_0,t_0),t)\in X_c$ and, by \eqref{eq:subsol1} together with Theorem~\ref{th:dpp} (DPP),
    \begin{equation*}\label{eq:subsol2}
        - \frac{1}{t-t_0}\int_{t_0}^t \ell(m(s),a)\;ds -\frac{\varphi(m(t),t)-\varphi(m_0,t_0)}{t-t_0} \leq 0\qquad \forall t\in (t_0,t_0+\bar{h}).
    \end{equation*}
    Passing to the limit $t\downarrow t_0$, and using also Remark~\ref{rem:frechet}, we obtain
    \[
        -\ell(m_0,a) + \left\langle \nabla_m\varphi(m_0,t_0),\,\diver(am_0) \right\rangle - \partial_t\varphi(m_0,t_0) \leq 0.
    \]
    Since $a\in\tilde{\mathcal{A}}$ is arbitrary, it follows that
    $\displaystyle - \partial_t\varphi(m_0,t_0) + H(m_0,\nabla_m\varphi(m_0,t_0))\leq 0$.\\
    \emph{ii-a)} Let $\varphi\in C^1(L^2(\R^n)\times [0,T])$ and $(m_0,t_0)\in \bigcup_{t\in (0,T)} \left[ \bar{X}_c(t)\times\{t\} \right]$ be such that $V-\varphi$ attains a local minimum relative to $\bar{X}_c$ at $(m_0,t_0)$, i.e., there exists $r>0$ such that for all $(m,t)\in \bar{X}_c$ with $\|m-m_0\|_{L^2}<r$ and $ |t-t_0|<r$
    \begin{equation}\label{eq:supsol1}
        V(m_0,t_0) - \varphi(m_0,t_0) \leq V(m,t)-\varphi(m,t).
    \end{equation}
    Fix $\eps>0$ and $h\ll1$. By Theorem~\ref{th:dpp}, there exists $\bar{\alpha}\in\mathcal{A}(m_0,t_0)$ (throughout the proof, we use $\bar\alpha(s)$ for $\bar\alpha(\cdot,s)\in\tilde{\mathcal{A}}$), depending on $\eps$ and $h$, such that, setting $\bar{m}(s):=m(\cdot,s;\bar\alpha,m_0,t_0)\in \bar{X}_c(T)$, one has
    \[
        V(m_0,t_0) \geq \int_{t_0}^{t_0+h} \ell(\bar{m}(s),\bar\alpha(\cdot,s))\;ds + V(\bar{m}(t_0+h),t_0+h)-h \eps.
    \]
    Then, by \eqref{eq:supsol1}, it follows that,
    \begin{equation}\label{eq:supsol2}
         -\big[  \varphi(\bar{m}(t_0+h),t_0+h)-\varphi(m_0,t_0) \big]-\int_{t_0}^{t_0+h}\ell(\bar{m}(s),\bar{\alpha}(s))\;ds \geq -h\eps.
    \end{equation}
    By the Lipschitz continuity of $\ell$ in the state variable w.r.t the control and Remark \ref{remarkvitale}
    \[
        \int_{t_0}^{t_0+h} \ell(\bar m(s),\bar\alpha(s))\,ds
        =\int_{t_0}^{t_0+h}\ell(m_0,\bar\alpha(s))\,ds+o(h)\qquad \text{for }h \downarrow 0.
    \]
    Moreover, applying Remark~\ref{rem:frechet} to $\varphi$ along the trajectory $s\mapsto \bar m(s)$ and using the continuity of
    $\nabla_m\varphi$ and $\partial_t\varphi$, a standard decomposition yields, for $h \downarrow 0$,
    \[
        \varphi(\bar m(t_0+h),t_0+h)-\varphi(m_0,t_0)
        =-\int_{t_0}^{t_0+h} \left\langle \nabla_m\varphi(m_0,t_0),\diver\big(\bar\alpha(s)m_0\big)\right\rangle\,ds
        +h \partial_t\varphi(m_0,t_0)+o(h).
    \]
    Substituting these expansions into \eqref{eq:supsol2} and dividing by $h$, we obtain, for $h \downarrow 0$,
    \[
        -\partial_t\varphi(m_0,t_0)
        +\frac{1}{h}\int_{t_0}^{t_0+h}
        \Big[
            \left\langle \nabla_m\varphi(m_0,t_0),\diver\big(\bar\alpha(s)m_0\big)\right\rangle
            -\ell(m_0,\bar\alpha(s))
        \Big]ds
        \geq -\eps+o(1).
    \]
    Since $\displaystyle\left\langle \nabla_m\varphi(m_0,t_0),\diver\big(\bar\alpha(s)m_0\big)\right\rangle -\ell(m_0,\bar\alpha(s))\leq H\big(m_0,\nabla_m\varphi(m_0,t_0)\big)$ for all $s\in[t_0,t_0+h]$, we infer
    \[
        -\partial_t\varphi(m_0,t_0)+H\big(m_0,\nabla_m\varphi(m_0,t_0)\big)\geq -\eps+o(1)
        \qquad \text{for }h\downarrow 0.
    \]
    Letting $h \downarrow 0$ and then $\eps\downarrow 0$, we conclude that
    $-\partial_t\varphi(m_0,t_0)+H\big(m_0,\nabla_m\varphi(m_0,t_0)\big)\ge 0$.
\end{proof}

\section{Comparison principle and uniqueness}
\label{sec:comparison}

In this section we prove that, under our hypotheses, the value function is indeed the unique Lipschitz continuous constrained viscosity solution of \eqref{eq:HJB}.
\begin{lemma}\label{lemma:inward-cone}
    There exists a constant $B>0$ such that for every $(\bar{m},\bar{t})\in \bar{X}_c(T)\times [0,T)$,
    \[ 
        \bar{X}_c(T)\cap B_{L^2}(m(\cdot,\bar{t}+\eps;\zeta a_p,\bar{m},\bar{t}),\zeta\eps B) \subseteq X_c(T) \qquad \;\forall\zeta\in (0,1],\forall \eps \in (0,\min\{B,T-\bar{t}\}],
    \]
    where $a_p\in\tilde{\mathcal{A}}$ is the pushback control defined in Subsection~\ref{subsec:regularity}. 
\end{lemma}
\noindent  Note that, in general, we are not requiring that $(\bar m,\bar t)\in \bar X_c$. Lemma \ref{lemma:inward-cone} means that starting from $(\bar m,\bar t)\in\bar{X}_c(T)\times [0,T)$, i.e. $\langle p,\bar m\rangle\le\delta$, with a suitable rescaled pushback control, at least for immediately subsequent and close times $t$, one can center at each point $m(t)$ of the trajectory a $L^2$-
ball with radius growing linearly in time, while remaining entirely contained in $X_c(T)$, that is $\langle p,m(t)\rangle<\delta$. Compare with the cone property for the finite-dimensional case in Soner \cite{Soner1986StateConstraintI}, and with the one in Cannarsa-Gozzi-Soner \cite{CannarsaGozziSoner1991HilbertHJB} for their infinite-dimensional problem. See also Remark \ref{rem:penalization}.

\begin{proof}[Proof of Lemma \ref{lemma:inward-cone}]
    \textcolor{red}{Since $\supp(\bar m)\subseteq\tilde{\Omega}(T)$ and $\|\zeta a_p\|_{L^\infty}\leq M$, the flow representation gives, for $0<\eps\leq T-\bar t$,}
    \[
        {\color{red}\supp\big(m(\cdot,s;\zeta a_p,\bar m,\bar t)\big)
        \subseteq\tilde{\Omega}(T+s-\bar t)\subseteq\tilde{\Omega}(2T)
        \qquad \forall s\in[\bar t,\bar t+\eps].}
    \]
    Let $L_\eta$ be the Lipschitz constant of $t\mapsto \langle p, m(\cdot,t;\alpha,\bar{m},\bar{t})\rangle$ (independent of $(\alpha,\bar m,\bar t)$, as noted in the proof of Lemma~\ref{lemma:remake soner}), let $d:=\sup\{ p(x) : x\in {\color{red}\tilde{\Omega}(2T)} \}$, and let $r_0$ be as in \eqref{eq:uniform}. Set
    \begin{equation}\label{eq:inward-cone1}
        B:=\min \left\{ \sqrt{\frac{\delta}{2\|p\|_{L^2({\color{red}\tilde{\Omega}(2T)})}}},\frac{\delta}{4L_\eta},\frac{\bar{M}r_0\delta}{4d\|p\|_{L^2({\color{red}\tilde{\Omega}(2T)})}} \right\}.
    \end{equation}
    Let $\tilde{m}\in \bar{X}_c(T)\cap B_{L^2}(m(\cdot,\bar{t}+\eps;\zeta a_p,\bar{m},\bar{t}),\zeta\eps B)$; we show that $\langle p,\tilde{m} \rangle < \delta$. \\
    Explicitly, using \eqref{eq:continuity} and the divergence theorem as in the proof of Lemma~\ref{lemma:remake soner} (see the second equalities in \eqref{eq:sopra}),
    \begin{align}
        \nonumber \langle p,\tilde{m} \rangle & = \int_{\R^n} p(x)[\tilde{m}(x)-m(x,\bar{t}+\eps;\zeta a_p,\bar{m},\bar{t})]\;dx + \langle p,m(\cdot,\bar{t}+\eps;\zeta a_p,\bar{m},\bar{t})\rangle\\ 
        \nonumber & \leq \|p\|_{L^2({\color{red}\tilde{\Omega}(2T)})}\| \tilde{m}-m(\cdot,\bar{t}+\eps;\zeta a_p,\bar{m},\bar{t}) \|_{L^2(\R^n)} + \langle p,\bar{m}\rangle +\\
        \nonumber & \phantom{\leq}\, + \int_{\bar{t}}^{\bar{t}+\eps} \int_{{\color{red}\tilde{\Omega}(2T)}\setminus \Omega} \nabla p(x) \cdot \zeta a_p(x)m(x,t;\zeta a_p,\bar{m},\bar{t})\;dx\;dt\\
        \label{eq:inward-cone2} & < \|p\|_{L^2({\color{red}\tilde{\Omega}(2T)})} \zeta \eps B + \langle p,\bar{m}\rangle - \zeta \bar{M} r_0 \int_{\bar{t}}^{\bar{t}+\eps}\int_{{\color{red}\tilde{\Omega}(2T)}\setminus \Omega} m(x,t;\zeta a_p,\bar{m},\bar{t})\;dx\;dt,
    \end{align}
    where in the last inequality we used \eqref{eq:pushback inequality}. \\ 
    If $\langle p,\bar{m}\rangle \leq \delta/2$, then it clearly follows from \eqref{eq:inward-cone2}, using \eqref{eq:inward-cone1}, that
    \[ 
          \langle p,\tilde{m} \rangle  < \|p\|_{L^2({\color{red}\tilde{\Omega}(2T)})} B^2 + \frac{\delta}{2} \leq \delta.
    \]
    If instead $\langle p,\bar{m}\rangle > \delta/2$, then by the definition of $L_\eta$ and since $\eps \leq \delta/(4L_\eta)$, we have, similarly arguing as in \eqref{eq:sopra2}, that $\displaystyle \int_{{\color{red}\tilde{\Omega}(2T)\setminus\Omega}}m(x,t;\zeta a_p,\bar{m},\bar{t})\;dx\geq \frac{\delta}{4d}$ for every $t\in [\bar{t},\bar{t}+\eps]$. It then follows from \eqref{eq:inward-cone2} that
    \[  
        \langle p,\tilde{m}\rangle < \|p\|_{L^2({\color{red}\tilde{\Omega}(2T)})}\zeta \eps B + \delta -\zeta \bar{M}r_0 \eps \frac{\delta}{4d}\leq \delta,
    \]
    thanks to \eqref{eq:inward-cone1}.
\end{proof}

\begin{remark}\label{rem:adattato}
    By adapting the finite-dimensional argument in \cite[Ch.~II, Sec.~2, Lemma~2.10]{BardiCapuzzoDolcetta1997}, one can prove that, if $u$ satisfies Definition~\ref{def:solvisc}, item~\emph{i-a)}, then the same sentence still holds when the open interval $(0,T)$ is replaced by the semi-open interval $[0,T)$. The same applies to the definition of supersolution in item~\emph{ii-a)}. In other words, in both cases one may also allow $t_0=0$.
\end{remark}

\begin{theorem}\label{th:comparision}
    Let $u_1,u_2\in C(\bar{X}_c)$ be Lipschitz continuous (as $V$ in Theorem \ref{th:lipcontinuity-of-V}). Assume that $u_1$ is a viscosity subsolution of \eqref{eq:HJB} in $X_c$, while $u_2$ is a viscosity supersolution of \eqref{eq:HJB} in $\bar{X}_c$. Then
    \[u_1(m,t)-u_2(m,t)\leq 0\qquad \forall (m,t)\in\bar{X}_c.\]
\end{theorem}

\begin{proof}
    By a standard argument, fix $\gamma>0$ and set $u_\gamma:=u_1-\gamma(T-t)$; then $u_\gamma$ is a strict subsolution, namely at every interior test point as in Definition~\ref{def:solvisc}, \emph{i-a)}, one has $-\partial_t\varphi+H(m,\nabla_m\varphi)\le -\gamma$. Hence it is enough to prove $w_\gamma:=\sup_{\bar X_c}(u_\gamma-u_2)\le 0$, since this yields $u_1\le u_2+\gamma T$ on $\bar X_c$ and, letting $\gamma\downarrow 0$, $u_1\le u_2$. Fix $\gamma>0$, set $w:=w_\gamma$ and $u:=u_\gamma$, and argue by contradiction assuming $w>0$; let $(z,\tau)\in \bar{X}_c$ be such that $u(z,\tau)-u_2(z,\tau)=w$ (which exists by compactness of $\bar{X}_c$), and note that $\tau\neq T$ by \emph{i-b)} and \emph{ii-b)} in Definition~\ref{def:solvisc}.\\ 
    For every $\zeta\in(0,1]$ and $\eps\in(0,T-\tau]$, define the map $\Phi_{\zeta,\eps}:\bar{X}_c\times \bar{X}_c\to\R$ (see also Remark \ref{rem:penalization})
    \begin{align}
        \nonumber \Phi_{\zeta,\eps}(m_1,t_1,m_2,t_2) & := u(m_1,t_1)-u_2(m_2,t_2) - \left\|\frac{m_1-m_2}{\eps}-\frac{m(\cdot,\tau+\eps;\zeta a_p,z,\tau)-z}{\eps} \right\|_{L^2}^2\\ 
        \label{eq:penalization}& \phantom{:=}\, -\|m_2-z\|_{L^2}^2 -\left|\frac{t_1-t_2}{\eps}-1\right|^2-\eps(t_1+t_2),
    \end{align}
    where $a_p$ is as in $\eqref{eq:pushback-control-def}$. Let $(\bar{m}_1,\bar{t}_1,\bar{m}_2,\bar{t}_2)$ (which depends on $\zeta,\eps$) be a global maximum point of $\Phi_{\zeta,\eps}$ (which exists by compactness of $\bar{X}_c\times \bar{X}_c$). Adding $\pm u(z,\tau)$ and using Remark~\ref{remarkvitale},
    \begin{align}
        \nonumber \Phi_{\zeta,\eps}(m(\cdot,\tau+\eps;\zeta a_p,z,\tau),\tau+\eps,z,\tau) & = u(m(\cdot,\tau+\eps;\zeta a_p,z,\tau),\tau+\eps)-u_2(z,\tau)-\eps(2\tau+\eps)\\
        \nonumber & \geq -L_1(\|m(\cdot,\tau+\eps;\zeta a_p,z,\tau)-z\|_{L^2}+\eps)+w-\eps(2T+\eps)\\
        \label{confr1}& \geq -L_1(L_T\eps+\eps)+w-3T\eps,
    \end{align}
    where $L_1$ is the Lipschitz constant of $u$. Moreover, $u(\bar{m}_1,\bar{t}_1)-u_2(\bar{m}_2,\bar{t}_2) \leq L_1(\|\bar{m}_1-\bar{m}_2\|_{L^2}+|\bar{t}_1-\bar{t}_2|)+w$; therefore, since $(\bar{m}_1,\bar{t}_1,\bar{m}_2,\bar{t}_2)$ is a maximum point, using \eqref{confr1},
    \begin{align}
        \label{confr2} \bigg\|\frac{\bar{m}_1-\bar{m}_2}{\eps}&-\frac{m(\cdot,\tau+\eps;\zeta a_p,z,\tau)-z}{\eps} \bigg\|_{L^2}^2+\|\bar{m}_2-z\|_{L^2}^2 +\bigg|\frac{\bar{t}_1-\bar{t}_2}{\eps}-1\bigg|^2\leq \\ 
        \nonumber &\leq L_1(\|\bar{m}_1-\bar{m}_2\|_{L^2}+|\bar{t}_1-\bar{t}_2|+L_T\eps+\eps)+3T\eps\qquad \forall(\zeta,\eps)\in (0,1]\times (0,T-\tau]. 
    \end{align}
    Since $\bar{X}_c$ is bounded, the right-hand side of \eqref{confr2} is bounded by a constant $C_1>0$; in particular,
    \begin{enumerate}
        \renewcommand{\labelenumi}{(\roman{enumi})}
        \renewcommand{\theenumi}{(\roman{enumi})}
        \item \label{confr3} adding $\displaystyle \pm \frac{1}{\eps}(m(\cdot,\tau+\eps;\zeta a_p,z,\tau)-z)$ and using Remark~\ref{remarkvitale}, $\|\bar{m}_1-\bar{m}_2\|_{L^2} \leq \eps (\sqrt{C_1}+L_T)$,
        \item \label{confr4} adding $\pm1$, $|\bar{t}_1-\bar{t}_2|\leq \eps(\sqrt{C_1}+1) $.
    \end{enumerate}
    We claim that for $\zeta \in (0,1],\eps\ll 1$ one has $\bar{t}_1,\bar{t}_2\in [0,T)$. 
   Extending $f:=u-u_2$ to $\bar{X}_c(T)\times [0,T]$ and applying Proposition \ref{prop:app-new} to $ g(t):=\max_{m\in \bar{X}_c(T)} \left[u(m,t)-u_2(m,t)\right]$, we deduce that, since $g(T)\leq 0$, there exists $\tau_0\in (0,T)$ such that
\begin{equation}\label{mon}
    u(m,t)-u_2(m,t)\leq \frac{w}{8}
    \qquad \forall (m,t)\in \bar{X}_c : t\in [T-\tau_0,T].
\end{equation}
    Assume by contradiction that $\bar{t}_1\in [T-\tau_0,T]$. By \eqref{mon}, \ref{confr3}, and \ref{confr4} (here $L_2$ is the Lipschitz constant of $u_2$)
    \begin{align*}
        \Phi_{\zeta,\eps}(\bar{m}_1,\bar{t}_1,\bar{m}_2,\bar{t}_2)
        &\leq u(\bar{m}_1,\bar{t}_1)-u_2(\bar{m}_1,\bar{t}_1)+u_2(\bar{m}_1,\bar{t}_1)-u_2(\bar{m}_2,\bar{t}_2)\\
        &\leq \frac{w}{8}+L_2\big(\|\bar{m}_1-\bar{m}_2\|_{L^2}+|\bar{t}_1-\bar{t}_2|\big)\leq \frac{w}{4},
    \end{align*}
    for $\varepsilon$ sufficiently small. On the other hand, \eqref{confr1} gives, possibly after further decreasing $\varepsilon$,
    \[
        \Phi_{\zeta,\eps}(m(\cdot,\tau+\eps;\zeta a_p,z,\tau),\tau+\eps,z,\tau)
        \geq w-\eps\big[L_1(L_T+1)+3T\big]\geq \frac{w}{2},
    \]
    contradicting the maximality of $(\bar{m}_1,\bar{t}_1,\bar{m}_2,\bar{t}_2)$. Therefore $\bar{t}_1\in [0,T-\tau_0]$ and it easily follows from \ref{confr4} that $\bar{t}_2\neq T$.\\ 
    We now show that for every $\zeta \in (0,1]$ there exists $\hat{\eps}_\zeta$ such that $\bar{m}_1\in X_c(T)$ for all $\eps \in (0,\hat{\eps}_\zeta]$. Let $C_X$ be as in \eqref{eq:radice}, $\tilde{L}$ as in Remark~\ref{remarkvitale}, and $B$ as in Lemma~\ref{lemma:inward-cone}; let $\zeta \in(0,1]$ and choose $\hat{\eps}_\zeta$ such that
     \begin{equation}\label{app}
        L_1(\hat{\eps}_\zeta(\sqrt{C_1}+L_T)+\hat{\eps}_\zeta(\sqrt{C_1}+1)+L_T\hat{\eps}_\zeta+\hat{\eps}_\zeta)+3T\hat{\eps}_\zeta
        < \min\left\{ \left[\frac{B\zeta}{2}\right]^2,\left[\frac{B}{2M \tilde{L} C_X}\right]^4 \right\}.
    \end{equation}
    Adding $\pm \bar{m}_2\pm z\pm m(\cdot,\tau+\eps;\zeta a_p,\bar{m}_2,\tau)$ and using \eqref{app}, \eqref{eq:continuity}, Remark~\ref{remarkvitale} and \eqref{confr2},
    \begin{align*}
            \|\bar{m}_1-m(\cdot,\tau+\eps;&\zeta a_p,\bar{m}_2,\tau)\|_{L^2}  \leq  \eps \left\|\frac{\bar{m}_1-\bar{m}_2}{\eps}- \frac{m(\cdot,\tau+\eps;\zeta a_p,z,\tau) -z}{\eps}\right\|_{L^2}   +\\
            & \phantom{\leq}+ \| \bar{m}_2-m(\cdot,\tau+\eps;\zeta a_p,\bar{m}_2,\tau)-\big[ z-m(\cdot,\tau+\eps;\zeta a_p,z,\tau) \big]\|_{L^2}\\
            & < \frac{B\zeta}{2}\eps + \int_\tau^{\tau+\eps}\left\|\diver(\zeta a_p [ m(\cdot,t;\zeta a_p,\bar{m}_2,\tau)-m(\cdot,t;\zeta a_p,z,\tau)])\right\|_{L^2}\;dt  \\
            & \leq \frac{B\zeta}{2}\eps + \zeta \int_\tau^{\tau+\eps} \|a_p\|_{W^{1,\infty}}\|m(\cdot,t;\zeta a_p,\bar{m}_2,\tau)-m(\cdot,t;\zeta a_p,z,\tau)\|_{H^1}\;dt.\\ 
            & \leq  \frac{B\zeta}{2}\eps+\zeta M\eps \tilde{L}C_X\sqrt{\|\bar{m}_2-z\|_{L^2}}<B\eps\zeta,\qquad \forall \eps \in (0,\hat{\eps}_\zeta], 
    \end{align*}
    therefore $\bar{m}_1\in \bar{X}_c(T)\cap B_{L^2}(m(\cdot,\tau+\eps;\zeta a_p,\bar{m}_2,\tau),B\eps\zeta)\subseteq X_c(T)$ by Lemma~\ref{lemma:inward-cone}. Hence, possibly after rescaling, for every $\zeta\in(0,1]$ there exists $\hat{\eps}_{\zeta} >0$ such that
    \begin{equation}\label{questo}
        \bar{m}_1\in X_c(T),\;\bar{m}_2\in \bar{X}_c(T),\;\bar{t}_1\in [0,T),\;\bar{t}_2\in [0,T) \qquad \forall \zeta\in(0,1],\;\forall \eps \in(0,\hat{\eps}_{\zeta}].
    \end{equation}
    Let $\varphi_1\in C^1(L^2(\R^n)\times[0,T])$, defined as follows:
    \[
        \varphi_1(m,t):=u_2(\bar{m}_2,\bar{t}_2)
        +\left\|\frac{m-\bar{m}_2}{\eps}
        - \frac{m(\cdot,\tau+\eps;\zeta a_p,z,\tau)-z}{\eps} \right\|_{L^2}^2
        +\|\bar{m}_2-z\|_{L^2}^2
        +\left|\frac{t-\bar{t}_2}{\eps}-1\right|^2
        +\eps(t+\bar{t}_2).
    \]
    Clearly $u(\bar{m}_1,\bar{t}_1)-\varphi_1(\bar{m}_1,\bar{t}_1)=\Phi_{\zeta,\eps}(\bar{m}_1,\bar{t}_1,\bar{m}_2,\bar{t}_2)\geq \Phi_{\zeta,\eps}(m,t,\bar{m}_2,\bar{t}_2)=u(m,t)-\varphi_1(m,t)$ for all $(m,t)\in X_c$. Therefore, for $\zeta\in(0,1]$ and $\eps\in(0,\hat{\eps}_{\zeta}]$, \eqref{questo} gives $(\bar{m}_1,\bar{t}_1)\in \bigcup_{t\in[0,T)}\left[X_c(t)\times\{t\}\right]$, and the inequality above shows that $(\bar{m}_1,\bar{t}_1)$ is a maximum point of $u-\varphi_1$ on $X_c$. Hence, by the strict subsolution condition (Definition~\ref{def:solvisc}, item~\emph{i-a)}, cf. Remark~\ref{rem:adattato}), one has that $\partial_t\varphi_1(\bar{m}_1,\bar{t}_1)+H\big(\bar{m}_1,\nabla_m\varphi_1(\bar{m}_1,\bar{t}_1)\big)\leq -\gamma$, i.e.,
    \begin{equation}\label{sot}
        -\left[ \frac{2}{\eps}\left( \frac{\bar{t}_1-\bar{t}_2}{\eps}-1 \right)+\eps \right]
        + H\left( \bar{m}_1,
        \frac{2}{\eps}\left(\frac{\bar{m}_1-\bar{m}_2}{\eps}
        -\frac{m(\cdot,\tau+\eps;\zeta a_p,z,\tau)-z}{\eps} \right)\right)
        \leq -\gamma.
    \end{equation}
    On the other hand, define $\varphi_2 \in C^1(L^2(\R^n)\times[0,T])$ as follows:
     \[
        \varphi_2(m,t):=u(\bar{m}_1,\bar{t}_1)
        - \left\|\frac{\bar{m}_1-m}{\eps}
        - \frac{m(\cdot,\tau+\eps;\zeta a_p,z,\tau)-z}{\eps} \right\|_{L^2}^2
        -\|m-z\|_{L^2}^2
        -\left|\frac{\bar{t}_1-t}{\eps}-1\right|^2
        -\eps(\bar{t}_1+t),
    \]
    Arguing similarly, and using again \eqref{questo} together with the supersolution property of $u_2$, we obtain that for every $\zeta\in(0,1]$ and $\eps\in(0,\hat{\eps}_{\zeta}]$,
    \begin{equation*}\label{sop}
       -\left[ \frac{2}{\eps}\left( \frac{\bar{t}_1-\bar{t}_2}{\eps}-1 \right)-\eps \right]
       + H\left( \bar{m}_2,
       \frac{2}{\eps}\left(\frac{\bar{m}_1-\bar{m}_2}{\eps}
       -\frac{m(\cdot,\tau+\eps;\zeta a_p,z,\tau)-z}{\eps} \right)
       -2(\bar{m}_2-z) \right)
       \geq 0,
    \end{equation*}
    Subtracting the latter inequality from \eqref{sot}, we obtain that for every $(\zeta,\eps)\in (0,1]\times (0,\min\{ \hat{\eps}_\zeta,\gamma/4 \}]$,
    \begin{align}
        \nonumber & \Delta H_{\zeta,\eps}
        := H\left( \bar{m}_1,
        \frac{2}{\eps}\left(\frac{\bar{m}_1-\bar{m}_2}{\eps}
        -\frac{m(\cdot,\tau+\eps;\zeta a_p,z,\tau)-z}{\eps} \right)\right) +\\
        \label{contradiction} 
        & \quad
        - H\left( \bar{m}_2,
        \frac{2}{\eps}\left(\frac{\bar{m}_1-\bar{m}_2}{\eps}
        -\frac{m(\cdot,\tau+\eps;\zeta a_p,z,\tau)-z}{\eps} \right)
        -2(\bar{m}_2-z) \right)
        \leq -\gamma +2\eps
        \leq -\frac{\gamma}{2}.
    \end{align}
    Now, for every $(\zeta,\eps)\in (0,1]\times (0,T-\tau]$, using \eqref{eq:continuity}, we have
    \begin{align}
        \nonumber \bigg\| &\frac{m(\cdot,\tau+\eps;\zeta a_p,z,\tau)-z}{\eps} \bigg\|_{L^2}
         = \frac{1}{\eps} \int_\tau^{\tau+\eps}
        \left\| \diver(\zeta a_p m(\cdot,t;\zeta a_p,z,\tau)) \right\|_{L^2}\;dt\\
        \label{final}& \leq \zeta \frac{1}{\eps}\int_\tau^{\tau+\eps}
        \|a_p\|_{W^{1,\infty}}
        \|m(\cdot,t;\zeta a_p,z,\tau)\|_{H^1}\;dt \leq \zeta M\tilde{C}_U=:C_3 \zeta.
    \end{align}
    On the other hand, it follows from \eqref{confr2}, using \ref{confr3} and \ref{confr4}, that there exists $C_2>0$ such that
    \begin{equation}\label{final2}
        \bigg\|\frac{\bar{m}_1-\bar{m}_2}{\eps}-\frac{m(\cdot,\tau+\eps;\zeta a_p,z,\tau)-z}{\eps} \bigg\|_{L^2},\|\bar{m}_2-z\|_{L^2} \leq C_2\sqrt{\eps}.
    \end{equation}
    It follows from \eqref{final} and \eqref{final2} that $\|\bar{m}_1-\bar{m}_2\|_{L^2} \leq \eps(C_2\sqrt{ \eps} + C_3 \zeta)$. Finally, using Lemma \ref{lemma:semilip hamilton}, it follows from what was just shown and again \eqref{final}, \eqref{final2}, \eqref{eq:radice} and \ref{confr3},
    \begin{align}
        \nonumber |\Delta H_{\zeta,\eps}|
        & \leq M\tilde{C}_U C_2\sqrt{\eps}
        +\frac{2}{\eps} C_2\sqrt{\eps}MC_X
        \sqrt{\|\bar{m}_1-\bar{m}_2\|_{L^2}}
        +L_\ell \eps (\sqrt{C_1}+L_T)\\
        \label{eq:finalconfr}& = M\tilde{C}_U \sqrt{C_2\eps}
        +2C_2MC_X\sqrt{C_2\sqrt{ \eps} + C_3 \zeta}
        +L_\ell \eps (\sqrt{C_1}+L_T).
    \end{align}
    Hence $\displaystyle \lim_{\zeta\to0^+}\lim_{\eps\to0^+}\left|\Delta H_{\zeta,\eps}\right| \leq \lim_{\zeta\to 0^+} 2C_2MC_X\sqrt{C_3\zeta} =0 $, in contradiction with \eqref{contradiction}.
\end{proof}

\noindent As it is standard, from Theorem \ref{th:lipcontinuity-of-V}, Proposition \ref{prop:solvisc} and Theorem \ref{th:comparision}, we have the following uniqueness result.

\begin{corollary}
    The value function $V$ is the unique Lipschitz continuous constrained viscosity solution of \eqref{eq:HJB}.
\end{corollary}

\begin{remark}\label{rem:penalization}
Here we briefly comment on the penalization function \eqref{eq:penalization} for the double variable technique in the proof of Theorem \ref{th:comparision}. It is obviously inspired by the one in Soner \cite{Soner1986StateConstraintI}, whose goal is to detach the points of maximum from the boundary of the constraint, $\int_{\Omega^c}pm=\delta$ in our case. However, besides the infinite dimension,  there are some important differences. 
The third term in the right-hand side of \eqref{eq:penalization} differs from the corresponding one in \cite{Soner1986StateConstraintI} because here the problem is time-dependent and in particular the structure of the set $\bar{X}_c$ does not allow to consider pairs of the form $(z',\tau)$ for any small inward perturbation $z'$ of $z$. Hence we need to enter the region $\int_{\Omega^c}pm<\delta$ by an admissibile trajectory given by the modulated pushback control $\zeta a_p$. This implies the use of the inner cone property as in Lemma \ref{lemma:inward-cone} and the fifth term in the right-hand side of \eqref{eq:penalization} when, as in \eqref{confr1}, we evaluate in two different instants $\tau+\varepsilon$ and $\tau$. Moreover we also need the small modulating parameter $\zeta$ in order to handle, after sending $\varepsilon\to0^+$, the last term in \eqref{eq:finalconfr}. We need this last step because of the norms estimate \eqref{eq:radice} in $\bar{X}_c$.
\end{remark}

\section{Some comments about different formulations of the state constraint}\label{sec:other-constraints}

\noindent In this section we discuss why the specific weighted state constraint \eqref{eq:constraint} adopted in this paper is natural for our framework, both conceptually and technically.

\noindent A first natural candidate is the \emph{hard} support constraint
\begin{equation}\label{eq:hardonstraint}
     \supp(m(\cdot,t))\subseteq \overline{\Omega}
    \qquad \forall t\in[0,T],
\end{equation}
which for nonnegative densities is equivalent to require
\begin{equation}\label{eq:hardconstraint2}
    \int_{\R^n\setminus\Omega} m(x,t)\,dx = 0
    \qquad \forall t\in[0,T].
\end{equation}
The difficulty with \eqref{eq:hardonstraint} is that this formulation is not robust in the functional topologies used in the paper: convergence in $L^2(\R^n)$ does not control supports. For instance, one can have $m_k\to 0$ in $L^2$ while all $m_k$ have full support (for instance $m_k(x)=k^{-1}e^{-|x|^2}$), so support information is unstable under small $L^2$ perturbations.\\
Regarding, on the other hand, \eqref{eq:hardconstraint2}, we still have a problem: from the viewpoint of the constraint, every admissible state is always \emph{on the boundary}: a state with mass compactly concentrated inside $\Omega$ (e.g., sufficiently far from violation) and a state with mass concentrated arbitrarily close to violation (e.g., a concentration near $\partial\Omega$) are both encoded in the same way. This seems to be incompatible with the correction mechanism in Lemma~\ref{lemma:remake soner}, where one needs a quantitative slack variable $\eta$ that measures how close the state is to violating the constraint, so that the correction is activated only when needed and remains uniformly controlled. In particular, the gain mechanism underlying the estimate analogous to \eqref{eq:ferrara} loses its validity.


\noindent To overcome this difficulty, a natural relaxation of \eqref{eq:hardconstraint2} would be, for instance for probability densities, (for some fixed $\delta\in(0,1)$)
\begin{equation}\label{eq:hardconstraint3}
     \int_{\R^n\setminus\Omega} m(x,t)\,dx \leq \delta
    \qquad \forall t\in[0,T].
\end{equation}

\noindent This is certainly more practical than \eqref{eq:hardconstraint2} for detecting who is actually on the boundary of the constraint; however this is still not sufficient for \eqref{eq:ferrara}. The reason is that \eqref{eq:hardconstraint3} only measures how much mass is outside $\Omega$ but not where that mass is located. Therefore, mass slightly outside and mass very far from $\Omega$ are treated with the same weight. With bounded controls, this lack of spatial sensitivity prevents a uniform short-time pushback gain: it seems that one cannot, as made here, recover a pointwise estimate playing the role of \eqref{eq:pushback inequality}, and consequently  a uniform lower bound of the type used in \eqref{eq:ferrara}.

\noindent For this reason we choose a different relaxation of \eqref{eq:hardconstraint2}: the weighted constraint \eqref{eq:constraint}, namely
\begin{equation}\label{eq:hardconstraint4}
    \int_{\R^n\setminus\Omega} p(x)m(x,t)\,dx\leq\delta,
\end{equation}
with $p(x)=\dist(x,\Omega)$ (or more generally a weight with an associated pushback control, as in Remark~\ref{rem:generalization domain and weight}). This formulation indeed detects the spatial distribution of constraint-violating masses and differently penalizes them in terms of being slightly outside or far outside $\Omega$. This is exactly the information needed to recover uniform quantitative estimates in Lemma~\ref{lemma:remake soner}, in particular \eqref{eq:ferrara}. \textcolor{red}{Also note that \eqref{eq:hardconstraint3} can be seen as the $L^2(\R^n)$-scalar product between $m$ and the characteristic function $\chi_\Omega$ while \eqref{eq:hardconstraint4} is the scalar product between $m$ and $p$. Another difference between \eqref{eq:hardconstraint3} and \eqref{eq:hardconstraint4} is hence that $p\in H^1_{\operatorname{loc}}(\R^n)$, while $\chi_\Omega$ does not.}\\ 
Possible future studies could be the limit of the problem with constraint \eqref{eq:hardconstraint4} for $\delta \downarrow 0$ and hence, formally, the study of the problem with constraint 
\[
    \int_{\R^n\setminus\Omega} p(x)m(x,t)\,dx=0,
\]
which does not well detect the constraint-violating masses, similarly to \eqref{eq:hardconstraint2}. Another possibility is to recover the problem with constraint \eqref{eq:hardconstraint3} as the limit of problems with constraint $\eqref{eq:hardconstraint4}$ for an appropriate sequence of weights $\{p_k\}$.

\appendix
\setcounter{equation}{0}
\setcounter{theorem}{0}
\renewcommand{\theequation}{a\arabic{equation}}
\renewcommand{\thetheorem}{A\arabic{theorem}}
\renewcommand{\theHequation}{A\arabic{equation}}
\renewcommand{\theHtheorem}{A\arabic{theorem}}
\section{Appendix}
\label{app:flow-estimates}

\begin{proposition}\label{prop:appendix}
   Under Assumption~\ref{ass:alpha}, the following estimates hold uniformly for admissible controls $\alpha\in\mathcal{A}$.
   \begin{enumerate}[label=\textup{(A\arabic*)},font=\normalfont]
        \item \label{eq:flowregularity1} $\displaystyle\|\Phi_{t,s_1}-\Phi_{t,s_2}\|_{L^\infty(\R^n)}\leq M|s_1-s_2| \qquad \forall t\in [0,T),\;s_1,s_2\in[t,T].$
        \item \label{eq:flowregularity2} $\displaystyle \|D\Phi_{t,s}\|_{L^\infty(\R^n)},\|D\Psi_{t,s}\|_{L^\infty(\R^n)}\leq e^{M(s-t)}\leq e^{MT}\qquad \forall t\in [0,T),\;s\in [t,T]$.
        \item \label{eq:flowregularity3} $\displaystyle \|\Psi_{t,s_1}-\Psi_{t,s_2}\|_{L^\infty(\R^n)}\leq Me^{MT}|s_1-s_2|\qquad \forall t\in [0,T),\;s_1,s_2\in[t,T].$
        \item \label{eq:flowregularity4} $\displaystyle \|D^2\Phi_{t,s}\|_{L^\infty(\R^n)}\leq C_{\Phi,2}$ and $\|D^2\Psi_{t,s}\|_{L^\infty(\R^n)}\leq C_{\Psi,2} \qquad\forall t\in [0,T),\;s\in [t,T]$.
        \item \label{eq:flowregularity5} $\displaystyle e^{-nM(s-t)}\leq \det D\Phi_{t,s}(x)\leq e^{nM(s-t)}\qquad \forall x\in\R^n,\;t\in [0,T), \; s\in[t,T].$
        \item \label{eq:flowregularity7} Let $J_{t,s}:=(\det D\Phi_{t,s})^{-1}$. Then $\displaystyle \left\|\nabla J_{t,s}\right\|_{L^\infty(\R^n)}\leq \tilde{C}_J \qquad\forall t\in [0,T),\;s\in [t,T].$
        \item \label{eq:flowregularity8} $\displaystyle \left\|D^2J_{t,s}\right\|_{L^\infty(\R^n)}\leq C_J \qquad\forall t\in [0,T),\;s\in [t,T].$
   \end{enumerate}
   The positive constants $C_{\Phi,2}, C_{\Psi,2}, \tilde{C}_J,$ and $C_J$ depend only on $M$, $T$, and $n$.
\end{proposition}

\begin{proof}
   Fix $x\in\R^n$ and integrate $\partial_s\Phi_{t,s}(x)=\alpha(\Phi_{t,s}(x),s)$ on $[s_1,s_2]$; then
$ \displaystyle|\Phi_{t,s_2}(x)-\Phi_{t,s_1}(x)|\le \int_{s_1}^{s_2}\|\alpha(\cdot,s)\|_{L^\infty}\,ds\le M|s_2-s_1|.$ Taking the supremum over $x$ gives \ref{eq:flowregularity1}.\\ 
Set $Y(s):=D\Phi_{t,s}(x)$ and $Z(s):=Y(s)^{-1}=D\Phi_{t,s}(x)^{-1}$. Differentiating the flow in $x$ gives $\partial_s Y=D\alpha(\Phi_{t,s},s)Y$, $Y(t)=I$, hence $\|D\Phi_{t,s}\|_{L^\infty}\le e^{M(s-t)}$ by Grönwall. For $Z$, differentiate the identity $Y(s)Z(s)=I$ to get
\[
    0=(\partial_s Y)Z+Y(\partial_s Z)\Rightarrow 
    \partial_s Z=-Y^{-1}(\partial_s Y)Z=-Y^{-1}D\alpha(\Phi_{t,s},s)YZ=-Z\,D\alpha(\Phi_{t,s},s).
\]
Grönwall yields $\|Z(s)\|\le e^{M(s-t)}$, and since $D\Psi_{t,s}(y)=D\Phi_{t,s}(x)^{-1}$ with $y=\Phi_{t,s}(x)$, we get $\|D\Psi_{t,s}\|_{L^\infty}\le e^{M(s-t)}$, which concludes \ref{eq:flowregularity2}.\\ 
Fix now $y\in\R^n$ and set $w(s):=\Psi_{t,s}(y)$, so $\Phi_{t,s}(w(s))=y$ for all $s$. Define $F(s,x):=\Phi_{t,s}(x)$; then $F(s,w(s))=y$ and
\[
0=\frac{d}{ds}F(s,w(s))=\partial_s\Phi_{t,s}(w(s))+D\Phi_{t,s}(w(s))\,w'(s).
\]
Since $\partial_s\Phi_{t,s}(x)=\alpha(\Phi_{t,s}(x),s)$ and $\Phi_{t,s}(w(s))=y$, we obtain
\[
D\Phi_{t,s}(w(s))\,w'(s)=-\alpha(y,s).
\]
By the chain rule on $\Phi_{t,s}\circ\Psi_{t,s}=\mathrm{Id}$,
\[
D\Phi_{t,s}(w(s))\,D\Psi_{t,s}(y)=I,
\]
hence $w'(s)=-D\Psi_{t,s}(y)\,\alpha(y,s)$. Therefore, for $s_1<s_2$,
\[
|\Psi_{t,s_2}(y)-\Psi_{t,s_1}(y)|
\le \int_{s_1}^{s_2}\|D\Psi_{t,s}\|_{L^\infty}\,\|\alpha(\cdot,s)\|_{L^\infty}\,ds
\le Me^{MT}|s_2-s_1|,
\]
using \ref{eq:flowregularity2}. Taking $\sup_y$ yields \ref{eq:flowregularity3}.\\ 
For $D^2\Phi_{t,s}$, differentiating once more in $x$ the identity $\partial_s D\Phi_{t,s}=D\alpha(\Phi_{t,s},s)\,D\Phi_{t,s}$ yields a linear equation involving $D^2\alpha$ and $D\Phi_{t,s}$, and Grönwall with \ref{eq:flowregularity2} gives a uniform bound depending only on $M,T,n$. For $D^2\Psi_{t,s}$, differentiate twice the identity $\Phi_{t,s}(\Psi_{t,s}(y))=y$ to obtain a linear system for the second derivatives of $\Psi_{t,s}$ with coefficients depending on $D\Phi_{t,s}(\Psi_{t,s}(y))$. Solving it yields $\|D^2\Psi_{t,s}\|_{L^\infty}\le C_n\|D\Psi_{t,s}\|_{L^\infty}^3\|D^2\Phi_{t,s}\|_{L^\infty}$; using \ref{eq:flowregularity2} and the uniform bound on $D^2\Phi_{t,s}$ just proved gives \ref{eq:flowregularity4}.

\noindent Furthermore, since $\alpha(\cdot,\tau)\in W^{3,\infty}$ we can choose a representative with globally Lipschitz second derivatives, hence differentiating the flow equation in $x$ is legitimate up to order three for a.e. $x$. Iterating the same argument and using the bounds in \ref{eq:flowregularity2} and \ref{eq:flowregularity4}, one can also obtain a uniform bound for the third derivatives of the flow, namely
\begin{equation}
\label{eq:flowregularity3rd}
\|D^3\Phi_{t,s}\|_{L^\infty(\R^n)}\le C_{\Phi,3}.
\end{equation}
As in \cite[Eq.~(2.7) and the subsequent lines]{bagagiolo2024single}, differentiating the determinant along the flow gives
\[
\frac{d}{ds}\det D\Phi_{t,s}(x)=(\mathrm{div}\,\alpha)(\Phi_{t,s}(x),s)\,\det D\Phi_{t,s}(x),\qquad \det D\Phi_{t,t}(x)=1,
\]
hence
\[
\det D\Phi_{t,s}(x)=\exp\!\left(\int_t^s(\mathrm{div}\,\alpha)(\Phi_{t,\tau}(x),\tau)\,d\tau\right).
\]
Using $\|\mathrm{div}\,\alpha(\cdot,\tau)\|_{L^\infty}\le n\|D\alpha(\cdot,\tau)\|_{L^\infty}\le nM$ yields \ref{eq:flowregularity5}; in particular, for all $s\in[t,T]$,
$\|(\det D\Phi_{t,s})^{-1}\|_{L^\infty(\R^n)}\le e^{nM(s-t)}\le e^{nMT}.$\\ 
Let $s\in[t,T]$ and set $A(x):=D\Phi_{t,s}(x)$. By Jacobi’s formula,
\begin{equation}\label{docJacobi}
\partial_{x_i}\det A(x)=\det A(x)\,\mathrm{tr}\!\big(A(x)^{-1}\partial_{x_i}A(x)\big).
\end{equation}
Thus $|\partial_{x_i}\det A|\le n\,\det A\,\|A^{-1}\|\,\|\partial_{x_i}A\|$. Using
\ref{eq:flowregularity2} for $\|A^{-1}\|$, \ref{eq:flowregularity4} for $\|D^2\Phi_{t,s}\|$ (hence $\|\partial_{x_i}A\|$), and \ref{eq:flowregularity5} for $\|\det A\|$, we obtain that there exists $C_{\det,1}>0$ such that 
\[
\|\nabla(\det D\Phi_{t,s})\|_{L^\infty(\R^n)}\le C_{\det,1}.
\]
Moreover,  $\displaystyle \nabla J_{t,s}=-\frac{\nabla(\det A)}{(\det A)^2}$, so by \ref{eq:flowregularity5} and the previous bound, we obtain \ref{eq:flowregularity7}, i.e., 
\[
    \left\|\nabla J_{t,s}\right\|_{L^\infty(\R^n)} \leq C_{\det,1} e^{2nMT}=: \tilde{C}_J.
\] 
Set $\tilde{A}(x):=\det A(x)=\det D\Phi_{t,s}(x)$. Differentiating \eqref{docJacobi} with respect to $x_j$ gives
\begin{align*}
    \partial_{x_jx_i}^2 \tilde{A} &=\partial_{x_j}\tilde{A}\,\mathrm{tr}\big(A^{-1}\partial_{x_i}A\big)+\tilde{A}\,\partial_{x_j}\Big(\mathrm{tr}\big(A^{-1}\partial_{x_i}A\big)\Big)\\
    & = \tilde{A}\,\mathrm{tr}\big(A^{-1}\partial_{x_j}A\big)\,\mathrm{tr}\big(A^{-1}\partial_{x_i}A\big)-\tilde{A}\,\mathrm{tr}\big(A^{-1}(\partial_{x_j}A)A^{-1}(\partial_{x_i}A)\big)+\tilde{A}\,\mathrm{tr}\big(A^{-1}\partial_{x_jx_i}^2A\big)\\ 
    & \leq e^{nMT} n^2 e^{2MT} C_{\Phi,2}^2 + e^{nMT} n e^{2MT}C_{\Phi,2}^2 + e^{nMT}n e^{MT}C_{\Phi,3}=: C_{\det,2},
\end{align*}
where we used \ref{eq:flowregularity2}, \ref{eq:flowregularity4}, \eqref{eq:flowregularity3rd}, and \ref{eq:flowregularity5}.

\noindent Moreover, for every $i,j$,
\[
\partial_{x_ix_j}^2(J_{t,s})
=2\tilde{A}^{-3}\,(\partial_{x_i}\tilde{A})(\partial_{x_j}\tilde{A})-J_{t,s}^2\,\partial_{x_ix_j}^2\tilde{A}.
\]
Using the bounds on $\tilde{A}$ (from \ref{eq:flowregularity5}), on $\partial_{x_i}\tilde{A}$ (from \ref{eq:flowregularity7}), and on $\partial_{x_ix_j}^2\tilde{A}$ (from the estimate just obtained), we get a uniform constant $C_J$, i.e., \ref{eq:flowregularity8}.
\end{proof}


\begin{proposition}\label{prop:app-new}
Let $(X,\|\cdot\|)$ be a normed space, $T>0$, and $Y\subseteq X\times [0,T]$. Assume that $f:Y\to \mathbb{R}$ is $L$-Lipschitz, that is,
\[
|f(x,t)-f(y,s)|\leq L\bigl(\|x-y\|+|t-s|\bigr)
\qquad \forall (x,t),(y,s)\in Y.
\]
Then there exists an extension $\tilde f:X\times [0,T]\to \mathbb{R}$ such that $\tilde f=f$ on $Y$ and
\[
|\tilde f(x,t)-\tilde f(y,s)|
\leq L\bigl(\|x-y\|+|t-s|\bigr)
\qquad \forall (x,t),(y,s)\in X\times[0,T].
\]
Furthermore if $X$ is compact then the map $\displaystyle g:[0,T] \to \R,\; t\mapsto \max_{x\in X} \tilde{f}(x,t)$ is continuous.
\end{proposition}

\begin{proof}
Define $\displaystyle \tilde f(x,t):=\inf_{(z,\tau)\in Y}\Bigl(f(z,\tau)+L\bigl(\|x-z\|+|t-\tau|\bigr)\Bigr)$. Clearly $\tilde{f}|_{Y}\equiv f$. Furthermore, let $(x,t),(y,s)\in X\times[0,T]$; clearly 
\[
    f(z,\tau)+L(\|x-z\|+|t-\tau|)\leq f(z,\tau)+L(\|x-y\|+\|y-z\|+|t-s|+|s-\tau|).
\]
Taking the infimum over $(z,\tau)\in Y$, we get $\tilde f(x,t)\leq \tilde f(y,s)+L\bigl(\|x-y\|+|t-s|\bigr),$ and exchanging the roles of $(x,t),$ $(y,s)$ we also obtain $
\tilde f(y,s)\leq \tilde f(x,t)+L\bigl(\|x-y\|+|t-s|\bigr)$. Hence
\[
|\tilde f(x,t)-\tilde f(y,s)|
\leq L\bigl(\|x-y\|+|t-s|\bigr).
\]
Let now $t,s\in[0,T]$. For every $x\in X$ we have that $\tilde f(x,t)\le \tilde f(x,s)+L|t-s|,$ and taking the supremum over $x\in X$, we obtain $g(t)\le g(s)+L|t-s|$.
Exchanging the roles of $t$ and $s$, we also get $g(s)\le g(t)+L|t-s|$, and therefore $|g(t)-g(s)|\le L|t-s|$.
\end{proof}

\bibliographystyle{abbrvurl}
\bibliography{BibliographyV2}

@article{Ambrosio_Crippa_2014, 
title={Continuity equations and ODE flows with non-smooth velocity}, 
volume={144}, 
DOI={10.1017/S0308210513000085}, 
number={6}, 
journal={Proceedings of the Royal Society of Edinburgh: Section A Mathematics}, 
author={Ambrosio, Luigi and Crippa, Gianluca}, 
year={2014}, 
pages={1191–1244}
}

@book{BardiCapuzzoDolcetta1997,
  author    = {Bardi, Martino and Capuzzo-Dolcetta, Italo},
  title     = {Optimal Control and Viscosity Solutions of Hamilton--Jacobi--Bellman Equations},
  series    = {Systems \& Control: Foundations \& Applications},
  publisher = {Birkh{\"a}user},
  address   = {Boston},
  year      = {1997},
  isbn      = {978-0-8176-4755-1},
  doi       = {10.1007/978-0-8176-4755-1}
}

@incollection{ambrosio2008flow,
  author       = {Luigi Ambrosio and Gianluca Crippa},
  title        = {Existence, uniqueness, stability and differentiability properties of the flow associated to weakly differentiable vector fields},
  booktitle    = {Transport Equations and Multi-D Hyperbolic Conservation Laws},
  series       = {Lecture Notes of the Unione Matematica Italiana},
  volume       = {5},
  publisher    = {Springer},
  year         = {2008},
  address      = {Berlin, Heidelberg},
  pages        = {3--57}
}

@book{Santambrogio2015OptimalTransport,
  author    = {Santambrogio, Filippo},
  title     = {Optimal Transport for Applied Mathematicians},
  publisher = {Birkh{\"a}user},
  address   = {Cham},
  year      = {2015},
  doi       = {10.1007/978-3-319-20828-2}
}

@article{LasryLions2007MFGI,
  author  = {Lasry, Jean-Michel and Lions, Pierre-Louis},
  title   = {Mean field games},
  journal = {Japanese Journal of Mathematics},
  volume  = {2},
  pages   = {229--260},
  year    = {2007},
  doi     = {10.1007/s11537-007-0657-8}
}

@misc{Cardaliaguet2013NotesMFG,
  author       = {Cardaliaguet, Pierre},
  title        = {Notes on Mean Field Games},
  howpublished = {Lecture notes},
  year         = {2013},
  note         = {From P.-L. Lions' lectures at Coll\`ege de France; version dated January 15, 2012},
  url          = {https://www.ceremade.dauphine.fr/%7Ecardaliaguet/MFG20130420.pdf}
}

@article{bagagiolo2024single,
  author    = {Fabio Bagagiolo and Rossana Capuani and Luciano Marzufero},
  title     = {A single player and a mass of agents: A pursuit evasion-like game},
  journal   = {ESAIM: Control, Optimisation and Calculus of Variations},
  volume    = {30},
  pages     = {17},
  year      = {2024},
  doi       = {10.1051/cocv/2024009},
  url       = {https://www.esaim-cocv.org/articles/cocv/pdf/2024/01/cocv2024009.pdf},
  publisher = {EDP Sciences, SMAI}
}

@article{bagagioloCapuaniMarzufero2025zerosum,
  author  = {Bagagiolo, Fabio and Capuani, Rossana and Marzufero, Luciano},
  title   = {A zero-sum differential game for two opponent masses},
  journal = {Partial Differential Equations and Applications},
  volume  = {6},
  pages   = {19},
  year    = {2025},
  doi     = {10.1007/s42985-025-00322-5}
}

@book{Grafakos2014ModernFourierAnalysis,
  author    = {Grafakos, Loukas},
  title     = {Modern Fourier Analysis},
  edition   = {3},
  series    = {Graduate Texts in Mathematics},
  volume    = {250},
  publisher = {Springer},
  address   = {New York},
  year      = {2014},
  isbn      = {978-1-4939-1229-2},
  doi       = {10.1007/978-1-4939-1230-8}
}

@article{CannarsaGozziSoner1991HilbertHJB,
  author  = {Cannarsa, Piermarco and Gozzi, Fausto and Soner, H. Mete},
  title   = {A boundary value problem for Hamilton--Jacobi equations in {H}ilbert spaces},
  journal = {Applied Mathematics \& Optimization},
  volume  = {24},
  number  = {1},
  pages   = {197--220},
  year    = {1991},
  doi     = {10.1007/BF01447742}
}

@article{Soner1986StateConstraintI,
  author  = {Soner, H. Mete},
  title   = {Optimal Control with State-Space Constraint {I}},
  journal = {SIAM Journal on Control and Optimization},
  volume  = {24},
  number  = {3},
  pages   = {552--561},
  year    = {1986},
  doi     = {10.1137/0324032}
}

@article{CrandallLions1985HJInfiniteDimI,
  author  = {Crandall, Michael G. and Lions, Pierre-Louis},
  title   = {Hamilton--Jacobi equations in infinite dimensions. {I}. Uniqueness of viscosity solutions},
  journal = {Journal of Functional Analysis},
  volume  = {62},
  number  = {3},
  pages   = {379--396},
  year    = {1985},
  doi     = {10.1016/0022-1236(85)90011-4}
}

@article{CrandallLions1986HJInfiniteDimII,
  author  = {Crandall, Michael G. and Lions, Pierre-Louis},
  title   = {Hamilton--Jacobi equations in infinite dimensions. {II}. Existence of viscosity solutions},
  journal = {Journal of Functional Analysis},
  volume  = {65},
  number  = {3},
  pages   = {368--405},
  year    = {1986},
  doi     = {10.1016/0022-1236(86)90001-5}
}

@article{JimenezMarigondaQuincampoix2020MultiagentWasserstein,
  author  = {Jimenez, Chlo{\'e} and Marigonda, Antonio and Quincampoix, Marc},
  title   = {Optimal control of multiagent systems in the {W}asserstein space},
  journal = {Calculus of Variations and Partial Differential Equations},
  volume  = {59},
  number  = {2},
  pages   = {58},
  year    = {2020},
  doi     = {10.1007/s00526-020-1718-6}
}

@article{Pogodaev2016OptimalControlContinuityEq,
  author  = {Pogodaev, Nikolay},
  title   = {Optimal control of continuity equations},
  journal = {Nonlinear Differential Equations and Applications NoDEA},
  volume  = {23},
  pages   = {21},
  year    = {2016},
  doi     = {10.1007/s00030-016-0357-2}
}

@article{CavagnariMarigondaQuincampoix2021CompatibilityConstraints,
  author  = {Cavagnari, Giulia and Marigonda, Antonio and Quincampoix, Marc},
  title   = {Compatibility of state constraints and dynamics for multiagent control systems},
  journal = {Journal of Evolution Equations},
  volume  = {21},
  pages   = {4491--4537},
  year    = {2021},
  doi     = {10.1007/s00028-021-00724-z}
}

@article{Daudin2023FokkerPlanckStateConstraints,
  author  = {Daudin, S{\'e}bastien},
  title   = {Optimal control of the {F}okker--{P}lanck equation under state constraints in the {W}asserstein space},
  journal = {Journal de Math{\'e}matiques Pures et Appliqu{\'e}es},
  volume  = {175},
  pages   = {37--75},
  year    = {2023},
  doi     = {10.1016/j.matpur.2023.05.002}
}

@article{mehlactil,
  author    = {Mohamed Mehdaoui and Deborah Lacitignola and Mouhcine Tilioua},
  title     = {State-Constrained Optimal Control of a Coupled Quasilinear Parabolic System Modeling Economic Growth in the Presence of Technological Progress},
  journal   = {Applied Mathematics \& Optimization},
  year      = {2025},
  volume    = {91},
  pages     = {14},
  doi       = {10.1007/s00245-024-10214-6}
}

@article{brupasaga,
      title={Emergence of meta-stable clustering in mean-field transformer models}, 
      author={Giuseppe Bruno and Federico Pasqualotto and Andrea Agazzi},
      year={2025},
      eprint={2410.23228},
      archivePrefix={arXiv}
 
}

@article{gesletpolrig,
  author    = {Borjan Geshkovski and Cyril Letrouit and Yury Polyanskiy and Philippe Rigollet},
  title     = {A Mathematical Perspective on Transformers},
  journal   = {Bulletin of the American Mathematical Society},
  volume    = {62},
  number    = {3},
  year      = {2025},
  doi       = {10.1090/bull/1863}
}

@article{AussedatJerhaouiZidani2024ViscosityCentralizedMeasure,
  author  = {Aussedat, Averil and Jerhaoui, Othmane and Zidani, Hasnaa},
  title   = {Viscosity solutions of centralized control problems in measure spaces},
  journal = {ESAIM: Control, Optimisation and Calculus of Variations},
  volume  = {30},
  pages   = {91},
  year    = {2024},
  doi     = {10.1051/cocv/2024081},
  note    = {HAL: hal-04335852}
}

@article{Bonnet2019PMPConstrainedWasserstein,
  author  = {Bonnet, Beno{\^i}t},
  title   = {A {Pontryagin} Maximum Principle in {W}asserstein spaces for constrained optimal control problems},
  journal = {ESAIM: Control, Optimisation and Calculus of Variations},
  volume  = {25},
  pages   = {52},
  year    = {2019},
  doi     = {10.1051/cocv/2019044}
}

@article{BonnetRossi2019PMPwasserstein,
  author  = {Bonnet, Beno{\^i}t and Rossi, Francesco},
  title   = {The {Pontryagin} Maximum Principle in the {W}asserstein Space},
  journal = {Calculus of Variations and Partial Differential Equations},
  volume  = {58},
  number  = {1},
  pages   = {7},
  year    = {2019},
  doi     = {10.1007/s00526-018-1447-2}
}

@article{frank2016,
title = {A relaxation result for state constrained inclusions in infinite dimension},
journal = {Mathematical Control and Related Fields},
volume = {6},
number = {1},
pages = {113-141},
year = {2016},
issn = {2156-8472},
doi = {10.3934/mcrf.2016.6.113},
url = {https://www.aimsciences.org/article/id/672b8e9a-1f1a-4226-b75a-2cd0e3092c4d},
author = {Frankowska, H{\'e}l{\`e}ne and Marchini, Elsa M. and Mazzola, Marco}
}

@article{frank2018,
  author  = {Frankowska, H{\'e}l{\`e}ne and Marchini, Elsa M. and Mazzola, Marco},
  title   = {Distance estimates for state constrained trajectories of infinite dimensional differential inclusions},
  journal = {ESAIM: Control, Optimisation and Calculus of Variations},
  volume  = {24},
  number  = {3},
  pages   = {1207--1229},
  year    = {2018},
  doi     = {10.1051/cocv/2017032}
}

\section*{Statements and Declarations}

\subsection*{Funding}
Fabio Bagagiolo was partially supported by the GNAMPA-INdAM 2026 project \emph{Analisi e controllo per alcuni problemi di evoluzione} (CUP E53C25002010001). Ivan Roman\`o was partially supported by the GNAMPA-INdAM 2026 project \emph{Problemi di controllo e mean-field con vincoli di stato e dimensione infinita: teoria, applicazioni} (CUP E53C25002010001).

\subsection*{Acknowledgments}
\textcolor{red}{The authors thank the anonymous reviewers for their valuable comments.}

\subsection*{Competing Interests}
The authors have no relevant financial or non-financial interests to disclose.

\subsection*{Data Availability}
Not applicable.

\subsection*{Code Availability}
Not applicable.

\subsection*{Ethics Approval}
Not applicable.

\subsection*{Author Contributions}
All authors contributed to the study conception and design.

\end{document}